\documentclass[reqno,dvips]{amsart}

\usepackage{mathrsfs}
\usepackage{amsmath,amssymb,amsaddr}
\usepackage{cite}
\usepackage[pdftex]{graphicx}

\usepackage{color}

\usepackage{comment}

\theoremstyle{plain}
\newtheorem{thm}{Theorem}[section]
\newtheorem{lem}[thm]{Lemma}

\newtheorem{dfn}[thm]{Definition}
\newtheorem{prop}[thm]{Proposition}
\newtheorem{rmk}[thm]{Remark}

\def\D{\mathrm{D}}
\def\M{\mathscr{M}}
\def\N{\mathscr{N}}

\def\h{\mathrm{h}}
\def\p{\mathrm{p}}
\def\s{\mathrm{s}}
\def\u{\mathrm{u}}

\def\Nset{\mathbb{N}}

\def\Rset{\mathbb{R}}
\def\Sset{\mathbb{S}}
\def\Tset{\mathbb{T}}
\def\Zset{\mathbb{Z}}

\DeclareMathOperator{\sech}{sech}
\DeclareMathOperator{\csch}{csch}
\DeclareMathOperator{\sn}{sn}
\DeclareMathOperator{\cn}{cn}
\DeclareMathOperator{\dn}{dn}

\def\id{\mathrm{id}}

\def\epsilon{\varepsilon}

\makeatletter
 \@addtoreset{equation}{section}
\makeatother
\def\theequation{\arabic{section}.\arabic{equation}}

\begin{document}


\title[First integrals and commutative vector fields]%
{Persistence of periodic and homoclinic orbits, first integrals and commutative vector fields
 in dynamical systems}

\author[S. Motonaga and K. Yagasaki]{Shoya Motonaga and Kazuyuki Yagasaki}

\address{Department of Applied Mathematics and Physics, Graduate School of Informatics,
Kyoto University, Yoshida-Honmachi, Sakyo-ku, Kyoto 606-8501, JAPAN}

\email{mnaga@amp.i.kyoto-u.ac.jp}
\email{yagasaki@amp.i.kyoto-u.ac.jp}

\date{\today}
\keywords{%
Periodic orbit; homoclinic orbit; first integral; commutative vector field; perturbation; 
 Melnikov's method}

\begin{abstract}
We study persistence of periodic and homoclinic orbits,
 first integrals and commutative vector fields in dynamical systems
 depending on a small parameter $\epsilon>0$
 and give several necessary conditions for their persistence.
Here we treat homoclinic orbits not only to equilibria but also to periodic orbits.
We also discuss some relationships {of these results
 with the standard subharmonic and homoclinic Melnikov methods
 for time-periodic perturbations of single-degree-of-freedom Hamiltonian systems,
 and with another version of the homoclinic Melnikov method
 for autonomous perturbations of multi-degree-of-freedom Hamiltonian systems.}
In particular, we show that
 a first integral which converges to the Hamiltonian {or another first integral}
 as the perturbation tends to zero
 does not exist near the unperturbed periodic or homoclinic orbits
 {in the perturbed systems}
 if the subharmonic or homoclinic Melnikov functions are not identically zero
 on connected open sets.
We illustrate our theory for {four} examples:
The periodically forced Duffing oscillator,
 {two identical pendula coupled with a harmonic oscillator,}
 a periodically forced rigid body and a three-mode truncation of a buckled beam.
\end{abstract}
\maketitle


\section{Introduction}
Let $\M$ be an $n$-dimensional paracompact oriented $C^4$ real manifold for $n\ge 2$.
{Here we require its paracompactness and orientedness for defining integrals on $\M$.}
Consider dynamical systems of the form
\begin{align}\label{sys}
 \dot{x}=X_\varepsilon(x), \quad x\in \M,
\end{align}
where $\epsilon$ is a small parameter such that $0<\varepsilon\ll1$,
 and the vector field $X_\varepsilon$ is $C^3$ with respect to $x$ and $\epsilon$.
Let $X_\varepsilon(x)=X^0(x)+\varepsilon X^1(x)+O(\varepsilon^2)$
 for $\epsilon>0$ sufficiently small.
The system \eqref{sys} becomes
\begin{align}\label{sys0}
\dot{x}=X^0(x)
\end{align}
when $\varepsilon=0$, and it is regarded as a perturbation of \eqref{sys0}.
Assume that the unperturbed system \eqref{sys0} has a periodic or homoclinic orbit
 and a first integral or commutative vector field.
Here we are mainly interested in their persistence 
 in \eqref{sys} for $\epsilon>0$ sufficiently small.

Bogoyavlenskij \cite{B98} extended a concept of Liouville integraility{\cite{A89,M99}},
which is defined for Hamiltonian systems,
and proposed a definition of integrability for general systems.
For \eqref{sys}, its integrability means that there exist $k\,(\ge 1)$ 
 commutative vector fields containing $X_\epsilon$ and $n-k\,(\ge 0)$ 
 first integrals
 for them such that the vector fields and first integrals are, respectively,
 linearly and functionally independent over a dense open set in $\M$.
For integrable systems in this meaning,
 we have a statement similar to the Liouville-Arnold theorem for Hamiltonian systems
 (e.g., Section~49 in Chapter 10 of \cite{A89}):
The flow on a level set of the first integrals is diffeomorphically conjugate to a linear flow
 on the $k$-dimensional torus $\mathbb{T}^k$
 if the level set is a $k$-dimensional compact manifold (see Proposition~2 of \cite{B98}).
Thus, the existence of first integrals and commutative vector fields
is closely related to integrability of \eqref{sys}. 

Even if the unperturbed system \eqref{sys0} is integrable,
 the perturbed system \eqref{sys} is generally believed to be nonintegrable
 for $\epsilon>0$ small.
For example, when the system \eqref{sys} is analytic and Hamiltonian for $\epsilon\ge 0$,
 a famous result of Poincar\'e \cite{P92} says that
 its analytic Liouville integrability does not persist for $\epsilon>0$
 under some generic assumptions.
This means that not only first integrals independent of the Hamiltonian
 but also (Hamiltonian) vector fields commutative with {$X_0$} do not persist in general.
See also \cite{K83} {for a more general result on nonexistence of first integrals,
 which was extended to non-Hamiltonian systems in \cite{K96}.}
Moreover, Morales-Ruiz \cite{M02} studied time-periodic Hamiltonian perturbations
 of single-degree-of-freedom Hamiltonian systems with homoclinic orbits,
 and showed a relationship between their nonintegrability
 and 
 {a version due to Ziglin \cite{Z81} of the Melnikov method \cite{M63}}
 by taking the {time $t$ and small parameter $\epsilon$ as state variables}.
{Here the Melnikov method enables us
 to detect transversal self-intersection of complex separatrices of periodic orbits
 unlike the standard version \cite{GH83,M63,W90}.}
More concretely, under some restrictive conditions,
 he essentially proved that they are {meromorphically} nonintegrable in the Bogoyavlenskij sense
 if the {Melnikov functions} are not identically zero,
 when a generalization due to Ayoul and Zung \cite{AZ10}
 of the Morales-Ramis theory \cite{M99,MR01},
 which provides a sufficient condition for nonintegrability of dynamical systems, is applied.
See Section~4.1 below for more details.
On the other hand, to the authors' knowledge,
 the persistence of first integrals and commutative vector fields,
 {especially when the unperturbed system \eqref{sys0} is nonintegrable},
 in non-Hamiltonian systems has attracted little attention.

In this paper, we give several necessary conditions
 for persistence of periodic or homoclinic orbits, first integrals or commutative vector fields in \eqref{sys}.
{In particular}, we treat homoclinic orbits not only to equilibria but also to periodic orbits.
{
Moreover, we see that
 persistence of periodic or homoclinic orbits and first integrals or commutative vector fields near them
 have the same necessary conditions
 (cf. Theorems~\ref{thm:main2}-\ref{thm:main4}, \ref{thm:main3}, \ref{thm:main3o}, \ref{thm:main6}
 and \ref{thm:main6o}).
This indicates close relationships between the dynamics
 and geometry of the perturbed systems.}
We also discuss some relationships {of these results
 with the standard subharmonic and homoclinic Melnikov methods \cite{GH83,M63,W90,Y96}
 for time-periodic perturbations of single-degree-of-freedom Hamiltonian systems as in \cite{M02},
 and with another version of the homoclinic Melnikov method due to Wiggins \cite{W88}
 for autonomous {Hamiltonian} perturbations of multi-degree-of-freedom {integrable} Hamiltonian systems.
The subharmonic Melnikov method provides a sufficient condition
 for persistence of periodic orbits in the perturbed system:
If the subharmonic Melnikov functions have a simple zero, then such orbits persist.
{For the latter homoclinic Melnikov method, we restrict ourselves to the case
 in which the unperturbed systems have invariant manifolds consisting of periodic orbits
 to which there exist homoclinic orbits since only such a situation can be treated in our result,
 although the technique was developed for more general systems.}
These versions of the Melnikov methods are described shortly in Section~4 below.}
In particular, we show that
 a first integral which converges to the Hamiltonian {or another first integral} as $\epsilon\to 0$
 does not exist near the unperturbed periodic or homoclinic orbits
 in the perturbed system for $\epsilon>0$ sufficiently small
 if the subharmonic or homoclinic Melnikov functions are not identically zero
 on connected open sets.
We illustrate our theory for {four} examples:
 The periodically forced Duffing oscillator \cite{GH83,W90},
 {two identical pendula coupled with a harmonic oscillator},
 a periodically forced rigid body \cite{Y18}
 and a three-mode truncation of a buckled beam \cite{Y01}.
The persistence of first integrals is discussed in the first {and second examples},
 the persistence of a first integral and periodic orbits in the {third} one
and the persistence of commutaive vector fields in the {fourth} one.

The outline of this paper is as follows:
In Sections~\ref{fi} and \ref{cvf},
we present our main results for first integrals and commutative vector fields, respectively,
as well as for both of periodic and homoclinic orbits.
 For the reader's convenience, in Appendix~\ref{AppendixA}, 
we collect basic notions and facts on connections of vector bundles
and linear differential equations as auxiliary materials for Section~3.
In Section~\ref{Melnikov}, we describe some relationships of the main results
with the subharmonic and homoclinic Melnikov methods
when the unperturbed system \eqref{sys0} is a single-degree-of-freedom Hamiltonian system.
Finally, we give the {four} examples to illustrate our theory in Section~\ref{Application}.


\section{First Integrals}\label{fi}
In this section,
 we discuss persistence of periodic and homoclinic orbits and first integrals for \eqref{sys}.
In the discussion here, less smoothness of $\M$ and $X_\epsilon$ is needed:
 $\M$ and $X_\epsilon$ are $C^3$ and $C^2$, respectively.

\subsection{Periodic orbits}\label{Periodic orbits}
We begin with a case in which
the unperturbed system \eqref{sys0} has a periodic orbit in \eqref{sys0}.
We make the following assumptions on \eqref{sys0}:
\begin{itemize}
\setlength{\leftskip}{-1.2em}
\item[(A1)]
There exists a $T$-periodic orbit $\gamma(t)$ for some constant $T>0$ in \eqref{sys0};
\item[(A2)]
There exists a non-constant $C^3$ first integral $F(x)$ of \eqref{sys0}, i.e.,
\[
dF(X^0)=0,
\]
near $\Gamma=\{\gamma(t)\mid t\in[0,T)\}$.
\end{itemize}
Define
\begin{align}\label{integral1}
\mathscr{I}_{F,\gamma}:=\int_{0}^{T} dF(X^1)({\gamma(t)})dt.
\end{align}
We state our main results for persistence of periodic orbits and first integrals.

\begin{thm}
\label{thm:main2}
Assume that (A1) and (A2) hold.
If the perturbed system \eqref{sys} has a $T_\varepsilon$-periodic orbit
$\gamma_\varepsilon$ depending $C^2$-smoothly on $\varepsilon$
such that $T_0=T$ and $\gamma_0=\gamma$, 
then the integral $\mathscr{I}_{F,\gamma}$ must be zero.
\end{thm}

\begin{proof}
Assume that (A1) and (A2) hold and the system \eqref{sys}
 has a periodic orbit $\gamma_\varepsilon=\gamma+O(\epsilon)$ for $\epsilon>0$.
Since $\gamma_\varepsilon$ is a $T_\varepsilon$-periodic orbit in \eqref{sys}, we compute
\begin{align*}
\int_{0}^{T_\varepsilon}  dF(X_\varepsilon)({\gamma}_\varepsilon(t))dt
=F(\gamma_\varepsilon(T_\varepsilon))-F(\gamma_\varepsilon(0))=0.
\end{align*}
On the other hand, since $F$ is a first integral of $X^0$,
we have $dF(X^0)=0$, so that
\begin{equation*}
dF(X_\epsilon)(\gamma_\epsilon(t))=\epsilon dF(X^1)(\gamma(t))+O(\epsilon^2).
\end{equation*}
Since $T_\epsilon=T+O(\epsilon)$, we see by the above two equations that
\begin{align*}
\int_{0}^{T_\varepsilon}  dF(X_\varepsilon)({\gamma}_\varepsilon(t))dt
=\varepsilon\int_{0}^{T} dF(X^1)(\gamma(t))dt+O(\varepsilon^2)
=\varepsilon\mathscr{I}_{F,\gamma}+O(\varepsilon^2)=0.
\end{align*}
Thus, we obtain $\mathscr{I}_{F,\gamma}=0$.
\end{proof}

\begin{thm}
\label{thm:main1}
Assume that (A1) and (A2) hold.
If the perturbed system \eqref{sys} has a $C^3$ first integral $F_\varepsilon(x)$ 
depending $C^2$-smoothly on $\varepsilon$ near $\Gamma$ such that $F_0(x)=F(x)$,
then the integral $\mathscr{I}_{F,\gamma}$ must be zero.
\end{thm}

\begin{proof}
Assume that (A1) and (A2) hold and the system
 has a first integral $F_\epsilon=F+\varepsilon F^1+O(\varepsilon^2)$ near $\Gamma$.
Since $\gamma$ is a $T$-periodic orbit in \eqref{sys0}, we have
\begin{align}
\int_{0}^T dF_\epsilon(X_0)(\gamma(t))dt
=F_\varepsilon(\gamma(T))-F_\varepsilon(\gamma(0))=0.
\label{eqn:main2a}
\end{align}
On the other hand, since $dF_\varepsilon(X_\varepsilon)=0$ and 
\[
dF_\varepsilon(X_\varepsilon)
=dF_\varepsilon (X^0)+\varepsilon dF_\varepsilon (X^1)+O(\varepsilon^2),
\]
we have 
\begin{align}
dF_\varepsilon (X^0)=-\varepsilon dF_\varepsilon (X^1)+O(\varepsilon^2)
\label{eqn:main2b}
\end{align}
near $\Gamma$.
From \eqref{eqn:main2a} and \eqref{eqn:main2b} we obtain
\begin{align*}
\int_{0}^T  dF_\varepsilon(X^0)({\gamma}(t))dt
=&-\varepsilon\int_{0}^T  dF_\varepsilon (X^1)({\gamma}(t))dt+O(\varepsilon^2)\\
=&-\varepsilon\mathscr{I}_{F,\gamma}+O(\varepsilon^2)=0,
\end{align*}
which yields the desired result.
\end{proof}
Theorems~\ref{thm:main2} and \ref{thm:main1} mean that
if $\mathscr{I}_{F,\gamma}\neq0$, 
then neither the periodic orbit  $\gamma$ nor first integral $F$ persists
in \eqref{sys}  for $\epsilon>0$.

\subsection{Homoclinic orbits}\label{persistence-hom}
We next consider a case in which
the unperturbed system \eqref{sys0} has a homoclinic orbit to an equilibrium 
or to a periodic orbit in \eqref{sys0}.
Instead of (A1) and (A2), we assume the following on \eqref{sys0}:
\begin{itemize}
\setlength{\leftskip}{-1em}
\item[(A1')]
There exists a homoclinic orbit $\gamma^{\h}(t)$ to a $T$-periodic orbit $\gamma^{\p}(t)$
in \eqref{sys0};
\item[(A2')]
There exists a non-constant $C^3$ first integral $F(x)$ of \eqref{sys0}
 near $\Gamma^\h=\{\gamma^\h(t)\mid t\in\Rset\}\cup\Gamma^\p$,
 where $\Gamma^\p=\{\gamma^\p(t)\mid t\in[0,T)\}$.
\end{itemize}
In assumption~(A1') $\gamma^{\p}$ may be an equilibrium.
As seen below
we have statements similar to Theorems~\ref{thm:main2} and \ref{thm:main1} in this case
but another idea is needed for their proofs
since the situation is not simple when $\gamma^{\h}(t)$ is a homoclinic orbit
to a periodic orbit.

\begin{figure}
\includegraphics[scale=0.5,bb=0 0 437 340]{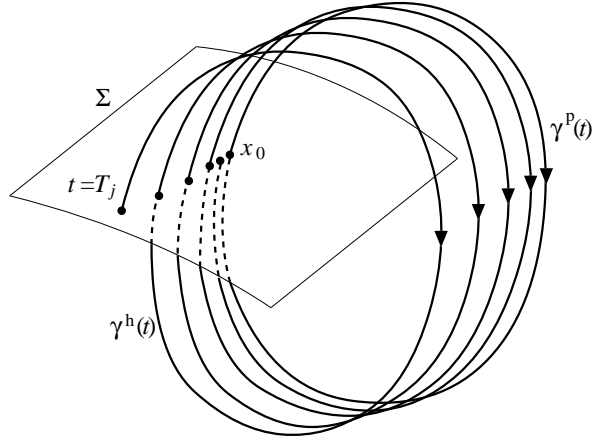}
\caption{Poincar\'e section $\Sigma$.
\label{fig:2a}}
\end{figure}

We first define an integral which plays a similar role
as $\mathscr{I}_{F,\gamma}$ in Section~\ref{Periodic orbits} (see Eq.~\eqref{integral1}).
%
Let $\gamma^\p$ be not an equilibrium. 
Choose a point $x_0=\gamma^{\p}(0)$
and take an $(n-1)$-dimensional hypersurface $\Sigma$ as the Poincar\'e section
such that $\gamma^\p$ intersects $\Sigma$ transversely at $x_0$.
Restricting $\Sigma$ to a sufficiently small neighborhood of $x_0$ if necessary, 
we can assume that $\gamma^{\h}(t)$ intersects $\Sigma$ transversely
infinitely many times, say at $T_j\in\Rset$ with $T_{j-1}<T_j$, $j\in\Zset$,
such that $\lim_{j\to-\infty}T_j=-\infty$ and $\lim_{j\to+\infty}T_j=+\infty$,
since it converges to $\gamma^{\p}(t)$.
{In particular},
\begin{align*}
\lim_{j\to \pm \infty} \gamma^\h(T_j)=x_0.
\end{align*}
See Fig.~\ref{fig:2a}.
So we formally define
 \begin{align}\label{Integ-hom-per}
\tilde{\mathscr{I}}_{F,\gamma^\h}
:=\lim_{k\to +\infty}\int_{T_{-k}}^{T_k} dF(X^1)({\gamma^{\h}(t)})dt.
\end{align}
If $\gamma_\epsilon^\p$ is an equilibrium,
then Eq.~\eqref{Integ-hom-per} is reduced to
\begin{equation}
\tilde{\mathscr{I}}_{F,\gamma^\h}
:=\int_{-\infty}^{\infty} dF(X^1)({\gamma^{\h}(t)})dt
\label{Integ-hom}
\end{equation}
by taking any sequence $\{T_j\}_{-\infty}^\infty$
 such that $\lim_{j\to\pm\infty}T_j=\pm\infty$.

We now state our main results for persistence of homoclinic orbits and first integrals.

\begin{thm}
\label{thm:main5}
Assume that (A1') and (A2') hold
and that there exists a periodic orbit $\gamma^\mathrm{p}_\varepsilon$
depending $C^2$-smoothly on $\varepsilon$ in \eqref{sys}
such that $\gamma^\mathrm{p}_0=\gamma^\mathrm{p}$.
If the perturbed system \eqref{sys} has a homoclinic orbit $\gamma^\h_\varepsilon$
to $\gamma_\epsilon^\mathrm{p}$ depending $C^2$-smoothly on $\varepsilon$
such that $\gamma^\h_0=\gamma^\h$,
then the limit in the right hand side of \eqref{Integ-hom-per} exists
and $\tilde{\mathscr{I}}_{F,\gamma^\h}=0$.
\end{thm}

\begin{proof}
Assume that the hypotheses of this theorem hold,
$\gamma^\p$ is not an equilibrium but periodic orbit,
and the system \eqref{sys} has a homoclinic orbit $\gamma^\h_\varepsilon=\gamma^\h+O(\epsilon)$
to a periodic orbit $\gamma_\epsilon^\p=\gamma^\p+O(\epsilon)$.
For $\epsilon>0$ sufficiently small,
 the periodic orbit $\gamma_\epsilon^\p$
 intersects the Poincar\'e section $\Sigma$ transversely, say at $t=0$.
Similarly, $\gamma^\h_\epsilon$ intersects $\Sigma$ transversely
infinitely many times, say at $T_j^\epsilon\in\Rset$ with $T_{j+1}^\epsilon<T_j^\epsilon$,
$j\in\Zset$,
such that $\lim_{j\to\pm\infty}T_j^\epsilon=\pm\infty$.
Moreover,
\[
\lim_{j\to\pm\infty} \gamma^\h(T^\epsilon_j)=x_\epsilon:=\gamma_\epsilon^\p(0).
\]
We easily see that
\begin{align}
\lim_{k \to +\infty} \int_{T^\epsilon_{-k}}^{T^\epsilon_k}  dF(X_\epsilon)({\gamma^\h_\epsilon}(t))dt
=&\lim_{k \to +\infty}\left(F({\gamma^\h_\varepsilon}(T_k^\varepsilon))
-F({\gamma^\h_\varepsilon}(T_{-k}^\varepsilon))\right)
=0.
\label{eqn:thm5a}
\end{align}

Introduce a metric in a neighborhood of $x_0$
 using the standard Euclidean one in the coordinates.
For $\delta>0$ sufficiently small,
 let $k>0$ be an integer such that $\gamma^\h(T_{\pm j})$ lie
 in a $\delta$-neighborhood of $x_0$ for $j>k$.
We can choose $\epsilon>0$ sufficiently small such that on $[T_{-k}^\epsilon,T_{k}^\epsilon]$
\[
\gamma_\epsilon^\h(t)=\gamma^\h(t)+O(\epsilon),
\]
which yields $T_j^\epsilon=T_j+O(\epsilon)$ for $|j|\le k$ and
\[
dF(X_\varepsilon)({\gamma^\h_\varepsilon}(t))
=\epsilon dF(X^1)({\gamma^\h}(t))+O(\epsilon^2)
\]
since $dF(X^0)=0$.
Hence,
\begin{equation}
\int_{T^\epsilon_{-k}}^{T^\epsilon_k}dF(X_\varepsilon)({\gamma^\h_\varepsilon}(t))dt
=\epsilon\int_{T^\epsilon_{-k}}^{T^\epsilon_k}dF(X^1)({\gamma^\h}(t))dt+O(\epsilon^2).
\label{eqn:thm5b}
\end{equation}
Taking $\delta\to 0$, we have $T_{\pm k}^\epsilon\to\pm\infty$,
 so that by \eqref{eqn:thm5a} and \eqref{eqn:thm5b}
the limit in the right hand side of \eqref{Integ-hom-per} exists and it must be zero.
\end{proof}

\begin{thm}
\label{thm:main4}
Assume that (A1') and (A2') hold.
If the perturbed system \eqref{sys} has a $C^3$ first integral $F_\varepsilon$
depending $C^2$-smoothly on $\varepsilon$ near $\Gamma^\h$ such that $F_0=F$,
then  the limit in the right hand side of \eqref{Integ-hom-per} exists
and $\tilde{\mathscr{I}}_{F,\gamma^\h}=0$.
\end{thm}

\begin{proof}
Assume that the hypotheses of the theorem hold,
$\gamma^\p$ is not an equilibrium but periodic orbit,
and the system \eqref{sys} has a first integral
$F_\varepsilon=F+\varepsilon F^1+O(\varepsilon^2)$ near $\Gamma^\h$.
We compute
\begin{align}
\lim_{k \to +\infty} \int_{T_{-k}}^{T_k}  dF_\varepsilon(X^0)({\gamma^\h}(t))dt
=&\lim_{k \to +\infty}\left(F_\varepsilon({\gamma^\h}(T_k))
-F_\varepsilon({\gamma^\h}(T_{-k}))\right)
=0.
\label{eqn:thm4a}
\end{align}
On the other hand, by \eqref{eqn:main2b}
\begin{equation}
\int_{T_{-k}}^{T_k}  dF(X^0)({\gamma^\h}(t))dt
=-\epsilon\int_{T_{-k}}^{T_k} dF(X^1)({\gamma^\h}(t))dt+O(\varepsilon^2).
\label{eqn:thm4b}
\end{equation}
As in the proof of Theorem~\ref{thm:main1},
it follows from \eqref{eqn:thm4a} and \eqref{eqn:thm4b}
that the limit in the right hand side of \eqref{Integ-hom-per} exists and it must be zero.
\end{proof}

\begin{rmk}
\hspace{1em}\\[-2.5ex]
\begin{enumerate}
\setlength{\leftskip}{-1.8em}
\item[(i)]
In the proofs of Theorems~\ref{thm:main5} and \ref{thm:main4},
 when $\gamma^\p$ is an equilibrium,
 we only have to choose any strictly monotonically increasing and diverging sequences
 $\{T_j^\epsilon\}_{-\infty}^\infty$, $\{T_j\}_{-\infty}^\infty$
 such that $T_j^\epsilon=T_j+O(\epsilon)$, $j\in\Zset$,
and to apply the same arguments.
\item[(ii)]
In Theorem~\ref{thm:main5}, if the periodic orbit (or equilibrium) $\gamma^\p$ is hyperbolic,
then the condition on existence of $\gamma_\epsilon^\p$ is not needed
since such a periodic orbit (or equlibrium) necessarily exists. 
\end{enumerate}
\end{rmk}

Theorems~\ref{thm:main5} and \ref{thm:main4} mean that
if $\tilde{\mathscr{I}}_{F,\gamma^\h}\neq0$, 
then neither the homoclinic orbit $\gamma^\h$
nor first integral $F$ persists in \eqref{sys} for $\epsilon>0$.

\section{Commutative Vector Fields}\label{cvf}
In this section, we 
discuss persistence of periodic and homoclinic orbits and commutative vector fields for \eqref{sys}.

\subsection{
Variational and adjoint variational equations}\label{VEAVE}
%
Before stating the main results, we give some preliminary results
on variational and adjoint variational equations.
A similar treatment in a complex setting are found in \cite{Ad, CR88,MR01}.
For the reader's convenience, some auxiliary materials are provided in Appendix A.

Let $\M$ be an $n$-dimensional paracompact oriented $C^3$ real manifold
as in Section~2.
Let $X$ be a $C^2$ vector field on $\M$
and let $\Gamma_\phi$ be an integral curve
given by a non-stationary 
solution $x=\phi(t)$ to the associated differential equation
\begin{equation}
\dot{x}=X(x).
\label{eqn:X}
\end{equation} 
The immersion $i:\Gamma_\phi\to\M$ induces a subbundle 
$T_{\Gamma_\phi}:=i^*T\M$ of the vector bundle $T\M$,
where $i^\ast$ represents the pullback of $i$.
Let $s:\Gamma_\phi\to T_{\Gamma_\phi}$ be a $C^1$ section of $T_{\Gamma_\phi}$.
We define the \emph{variational equation} (VE) of $X$ along $\Gamma_\phi$ as
\begin{equation}\label{VEconn}
\nabla s:=dt\otimes\mathcal{L}_XY|_{\Gamma_\phi}=0,
\end{equation}
where ``$\otimes$'' represents the tensor product,
$Y$ is any $C^1$ vector field extension of $s$ to $\M$,
$\mathcal{L}_X$ represents the Lie derivative along $X$,
and ``$dt\otimes$'' is frequently omitted in references.
Here $\nabla$ is a connection of $T_{\Gamma_\phi}$,
and $s$ is a horizontal section of $\nabla$
 if it satisfies the VE \eqref{VEconn}
 (see Appendix~A.1).
Locally, Eq.~\eqref{VEconn} is expressed as
\begin{equation}\label{VEloc}
\frac{d\hat{U}}{dt}=\frac{\partial X}{\partial x}(\phi(t))\hat{U}
\end{equation}
in the frame $\displaystyle
\left(\frac{\partial}{\partial x_1},\ldots,\displaystyle\frac{\partial}{\partial x_n}\right)$
associated with the coordinates $(x_1, ..., x_n)$,
where
\[
s=\sum_{j=1}^n \Xi_j \frac{\partial}{\partial x_j}.
\]
See Appendix~A.2.1 
for the derivation of \eqref{VEloc}.

Let $T^*_{\Gamma_\phi}$ be the dual bundle of $T_{\Gamma_\phi}$,
 and let $\alpha:\Gamma_\phi\to T^*_{\Gamma_\phi}$ be a $C^1$ section
 of $T^*_{\Gamma_\phi}$.

\begin{lem}
\label{lem:AVE}
The dual connection $\nabla^*$ of $\nabla$  (see Appendix~A.1) is given by
\begin{equation}
\nabla^* \alpha=\left.dt\otimes\mathcal{L}_X \omega\right|_{\Gamma_\phi},
\label{eqn:lemAVE}
\end{equation}
where $\omega:\M\to T^\ast\M$ is any $C^1$ differential $1$-form extension of $\alpha$.
\end{lem}

\begin{proof}
Let $s$ be a section of $T_{\Gamma_\phi}$and let $Y$ be its vector field extension as above.
The Lie derivative $\mathcal{L}_X$ satisfies
$$
\mathcal{L}_X \langle{Y, \omega}\rangle
=\langle{ \mathcal{L}_X Y, \omega}\rangle
+\langle{Y, \mathcal{L}_X \omega}\rangle,
$$
which yields
$$
\left.dt\otimes \mathcal{L}_X \langle{Y, \omega}\rangle\right|_{\Gamma_\phi}
=\langle{ \nabla s, \alpha}\rangle
+\langle{s,\left.dt\otimes\mathcal{L}_X \omega\right|_{\Gamma_\phi}}\rangle
$$
when restricted to $\Gamma_\phi$,
where $\langle\cdot,\cdot\rangle$ denotes the natural pairing by the duality.
On the other hand, since $\Gamma_\phi$ is an integral curve of $X$, we have
$$
d\langle{s, \alpha}\rangle
=\left.d\langle{Y, \omega}\rangle\right|_{\Gamma_\phi}
=\left.dt\otimes X(\langle{Y, \omega}\rangle)\right|_{\Gamma_\phi}
=\left.dt\otimes\mathcal{L}_X\langle{Y, \omega}\rangle\right|_{\Gamma_\phi}.
$$
By definition, we obtain \eqref{eqn:lemAVE}.
\end{proof}

We call
\begin{equation}\label{adjVEconn}
{\nabla}^* \alpha=0
\end{equation}
the \emph{adjoint variational equation} (AVE) of $X$ along $\Gamma_\phi$.
Thus, $\alpha$ is a horizontal section of $\nabla^\ast$
 if it satisfies the AVE \eqref{adjVEconn}.
Locally, Eq.~\eqref{adjVEconn} is expressed as
\begin{equation}\label{adjVEloc}
\frac{d\eta}{dt}=-\left( \frac{\partial X}{\partial x}(\phi(t))\right)^\mathrm{T} \eta
\end{equation}
in the frame $(dx_1,\ldots,d x_n)$,
where the superscript ``{\scriptsize $\mathrm{T}$}'' represents the transpose operator and
\[
\alpha=\sum_{j=1}^n \eta_j  dx_j.
\]
See Appendix~A.2.2 
for the derivation of \eqref{adjVEloc}.


\begin{lem}
\label{lem:dF}
{\rm(i)} If $X$ has a first integral $F$, then the following hold:
\begin{itemize}
\setlength{\leftskip}{-1.6em}
\item[(ia)]
The section $\alpha=dF|_{\Gamma_\phi}$ of $T^*_{\Gamma_\phi}$
 satisfies the AVE \eqref{adjVEconn} of $X$ along $\Gamma_\phi$;
\item[(ib)]
 $\langle s, dF|_{\Gamma_\phi}\rangle$ is a first integral
 of the VE \eqref{VEconn} of $X$ along $\Gamma_\phi$, i.e.,
 \[
 d\langle s, dF|_{\Gamma_\phi}\rangle=0
 \]
if the section $s$ of $T_{\Gamma_\phi}$ satisfies \eqref{VEconn}.
\end{itemize}
{\rm(ii)} If $X$ has a commutative vector field $Z$, i.e.,
\[
[X,Z]=0,
\]
where $[\cdot,\cdot]$ denotes the Lie bracket, then the following hold:
\begin{itemize}
\setlength{\leftskip}{-1.4em}
\item[(iia)]
The section $s=Z|_{\Gamma_\phi}$ of $T_{\Gamma_\phi}$ satisfies the VE \eqref{VEconn};
\item[(iib)]
$\langle Z|_{\Gamma_\phi}, \alpha\rangle$ is a first integral of the AVE \eqref{adjVEconn}, i.e.,
\[
d\langle Z|_{\Gamma_\phi}, \alpha\rangle=0
\]
if the section $\alpha$ of $T^*_{\Gamma_\phi}$ satisfies \eqref{adjVEconn}.
\end{itemize}
\end{lem}

\begin{proof}
Let $s$ and $\alpha$ satisfy the VE \eqref{VEconn} and AVE \eqref{adjVEconn}, respectively.
Since
\begin{align*}
d\langle {s,\alpha}\rangle 
=\langle {\nabla s, \alpha}\rangle+\langle{s,\nabla^* \alpha} \rangle
=\langle {0, \alpha}\rangle+\langle{s,0} \rangle
=0,
\end{align*}
we see that $\langle {s,\alpha}\rangle$ is a constant.
Hence, parts (ib) and (iib) immediately follow from (iia) and (ia), respectively.

Now we show (ia) and (iia).
If $X$ has a first integral $F$, then by Cartan's formula
(see, e.g., Theorem~4.2.3 of \cite{MR99})
we have
$$
\mathcal{L}_X dF=d(i_X (dF))+i_X(d^2 F)=0,
$$
where $i_X$ denotes the interior product of $X$.
This yields (ia) when restricted to $\Gamma_\phi$.
If $X$ has a commutative vector field $Z$, then we obtain (iia) 
since $\mathcal{L}_X Z|_{\Gamma_\phi}=[X,Z]|_{\Gamma_\phi}=0$.
\end{proof}

Similar results to Lemma~\ref{lem:dF} for symplectic connections
can be proven by using the musical isomorphism of symplectic forms 
(see Lemma~4.1 of \cite{MR01} and Chapter~4 of \cite{P92}).

\subsection{
Periodic orbits}

We turn to the issue of persistence of periodic orbits and commutative vector fields
 in \eqref{sys}.
Instead of (A2), we assume the following on \eqref{sys0}:
\begin{itemize}
\setlength{\leftskip}{-1.2em}
\item[(A3)]
The unperturbed system \eqref{sys0} has a $C^3$ commutative vector field $Z$, i.e.,
\[
[X^0,Z]=0,
\]
near $\Gamma$, such  that it is linearly independent of $X^0$.
\end{itemize}

\begin{lem}
\label{lem:adj}
Under assumption (A1),
 the connection $\nabla^\ast$ of $T^\ast_\Gamma$
 has a nontrivial horizontal section $\omega:\Gamma\to T^\ast_\Gamma$,
 i.e., $\omega$ satisfies the  AVE \eqref{adjVEconn} of $X^0$ along $\Gamma$.
\end{lem}

\begin{proof}
Let $\psi^t$ denote the flow of $X^0$ and let $x_0=\gamma(0)\in\Gamma$.
For any $p\in \Gamma$, there exists a unique time $t_p\in [0,T)$
such that $\gamma(t_p)=\psi^{t_p}(x_0)=p$.
Define $\theta:\Gamma \to [0,T)$ by $\theta(p):=t_p$.
 Since $\theta \circ \gamma=\id$, we have
\begin{align}\label{id}
\frac{d}{dt} \theta(\gamma(t))=1,
\end{align}
where $\id$ represents the identity map.
On the other hand, by the tubular neighborhood theorem
 (e.g., Theorem~5.2 in Chapter 4 of \cite{H76}),
there is a neighborhood $\mathscr{N}(\Gamma)$ of $\Gamma$
which is diffeomorphic to the normal bundle $N_\Gamma$ of $\Gamma$ in $\M$.
Let $f:\mathscr{N}(\Gamma)\to N_\Gamma$ be the diffeomorphism,
and let $\pi: N_\Gamma \to \Gamma$ be the natural projection.
Define a map $\Theta:\mathscr{N}(\Gamma)\to\Rset$ by $\Theta:=f^*\pi^*\theta$.
Since $f|_{\Gamma}=\mathrm{id}$ and $\pi|_{\Gamma}=\mathrm{id}$,
we have
\begin{equation}
\Theta|_{\Gamma}=\theta.
\label{eqn:lemadj}
\end{equation}
Using \eqref{id} and  \eqref{eqn:lemadj}, we show that for $x\in \Gamma$
\begin{align*}
(i_Xd\Theta)_x 
=(\mathcal{L}_X\Theta)_x
=\lim_{t\to 0}\frac{\Theta(\psi^t(x))-\Theta(x)}{t}
=\lim_{t\to 0}\frac{\theta(\psi^t(x))-\theta(x)}{t}
=1,
\end{align*}
which yields
\begin{align*}
\mathcal{L}_X(d\Theta)|_\Gamma=(i_Xd^2\Theta+di_Xd\Theta)|_\Gamma
=d((i_Xd\Theta)|_\Gamma)
=0
\end{align*}
by Cartan's formula.
Hence, by Lemma~\ref{lem:AVE} we see that
$\omega=d\Theta|_\Gamma$ is a nontrivial horizontal section of $\nabla^\ast$
 since $\theta$ is not a constant.
\end{proof}

\begin{rmk}
\label{rmk:dF}
\hspace{1em}\\[-2.5ex]
\begin{enumerate}
\setlength{\leftskip}{-1.8em}
\item[\rm(i)]
Let $\M=\Rset^n$.
We see that $\dot{\gamma}(t)$ is a periodic solution to the VE \eqref{VEloc}
 and consequently its Floquet exponents (see, e.g., Section~2.4 of \cite{C06}) include one.
Hence, the AVE \eqref{adjVEloc} possesses one as its Floquet exponent
 and consequently it has a periodic solution,
 which provides a horizontal section of $\nabla^\ast$ as guaranteed by Lemma~\ref{lem:adj}.
\item[(ii)]
Assume that (A1) and (A2) hold.
From Lemma~\ref{lem:dF} (ia) we see that
$dF|_\Gamma$ is a horizontal section of $\nabla^\ast$.
\end{enumerate}
\end{rmk}

Let $\omega$ be a horizontal section of $\nabla^\ast$ as stated in Lemma~\ref{lem:adj},
 and define the integral
\begin{align}\label{integral2}
\mathscr{J}_{\omega, Z, \gamma}:=\int_{0}^{T} \omega([X^1, Z])({\gamma(t)})dt.
\end{align}
We now state our result on persistence of commutative vector fields.

\begin{thm}
\label{thm:main3}
Assume that (A1) and (A3) hold.
If the perturbed system \eqref{sys} has a $C^3$ commutative vector field $Z_\varepsilon$
depending $C^2$-smoothly on $\varepsilon$ near $\Gamma$ such that $Z_0=Z$,
then the integral $\mathscr{J}_{\omega, Z,\gamma}$ 
is zero.
\end{thm}

For the proof of Theorems~\ref{thm:main3}
 we use the cotangent lift trick \cite{AZ10},
 and rewrite \eqref{sys} as a Hamiltonian system.
In this situation the persistence of commutative vector fields of \eqref{sys}
 is reduced to that of first integrals of the lifted Hamiltonian system.
We first explain the trick in a general setting, following \cite{AZ10}.
See Chapter~5 of \cite{MR99} for necessary information on Hamiltonian mechanics.

Let $T^*\!\M$ be the cotangent bundle of $\M$
and let $\pi:T^*\!\M\to \M$ be the natural projection.
Define a differential 1-form $\lambda:T^*\!\M\to T^*(T^*\!\M)$,
which is often called a \emph{(Poincar\'e-)Liouville form}, as
\[
\lambda_z=p(d\pi_z (\cdot)),
\]
where $z=(x, p)\in T^*\!\M$. 
Letting $\Omega_0=d\lambda$,
we have a symplectic manifold $(T^*\!\M, \Omega_0)$.
In the local coordinates $(x_1, ..., x_n, p_1, ...,p_n)$,
$\lambda$ and $\Omega_0$ are written as
\[
\lambda=\sum_{k=1}^n p_k dx_k\quad\mbox{and}\quad
\Omega_0=\sum_{k=1}^n dp_k \wedge dx_k,
\]
respectively.

Let $X$ be a smooth vector field on $\M$,
and define a function $h_X:T^*\!\M\to\Rset$ as
\begin{equation}\label{liftfcn}
h_X(x, p)=\langle {p, X(x)}\rangle,
\end{equation}
where $(x,p) \in T^*\!\M$. 
Then the Hamiltonian vector field $\hat{X}$ with the Hamiltonian $h_X$
 on the symplectic manifold $(T^*\!\M, \Omega_0)$
 is called the \emph{cotangent lift} of $X$.
Note that 
 the smoothness of $\hat{X}$ is less by one than that of $X$.
In the local coordinates $(x_1, ..., x_n, p_1, ...,p_n)$,
 the vector field $\hat{X}$ is expressed as
\begin{equation}
\frac{dx}{dt}=X(x) \left(=\frac{\partial h_X}{\partial p}\right),\quad
\frac{dp}{dt}=-\frac{\partial X(x)}{\partial x}^\mathrm{T}p
 \left(=-\frac{\partial h_X}{\partial x}\right),
\label{eqn:cl}
\end{equation}
the second equation of which has the same form
 as the AVE \eqref{adjVEloc} when $x=\phi(t)$.

\begin{lem}\label{lem:poisson}
For any vector fields $X$ and $Z$ on $\M$ we have
$$
\{h_X, h_Z\}=h_{[X, Z]}
$$
(see Eq.~\eqref{liftfcn}),
 where $\{\cdot,\cdot\}$ denotes the Poisson bracket for the symplectic form $\Omega_0$.
\end{lem}

\begin{proof}
In the local coordinates $(x_1, ..., x_n, p_1, ..., p_n)$, 
we write 
\[
X=\sum_{i=1}^n X_i \frac{\partial}{\partial x_i},\quad
Z=\sum_{j=1}^n Z_j \frac{\partial}{\partial x_j},\quad
p=\sum_{l=1}^n p_l\,dx^l.
\]
We compute
\begin{align*}
\{h_X, h_Z\}
=&\{\langle p, X(x)\rangle, \langle p, Z(x)\rangle\}
=\left\{\sum_{i=1}^n p_i X_i(x), \sum_{j=1}^n p_j Z_j(x)\right\}\\
=&\sum_{k=1}^n \left(X_k\sum_{i=1}^n p_i \frac{\partial Z_i}{\partial x_k}
-Z_k\sum_{j=1}^n p_j \frac{\partial X_j}{\partial x_k}\right)\\
=&\sum_{i=1}^n p_i \sum_{k=1}^n \left(X_k\frac{\partial Z_i}{\partial x_k}
-Z_k\frac{\partial X_i}{\partial x_k}\right)
=\langle p, [X, Z]\rangle
=h_{[X, Z]},
\end{align*}
which yields the desired result.
\end{proof}

We also need the following fact,
 which was used in the proof of Proposition~2 of \cite{AZ10}.

\begin{lem}
\label{lem:vec}
If $Z$ is a commutative vector field of $X$,
then $h_Z$ is a first integral for the cotangent lift $\hat{X}$ of $X$.
\end{lem}

\begin{proof}
It follows from Lemma~\ref{lem:poisson} that $dh_Z(\hat{X})=\{h_X, h_Z\}=h_{[X, Z]}$.
Hence, $dh_Z(\hat{X})=0$ if $[X, Z]=0$.
\end{proof}

We are now in a position to prove Theorem~\ref{thm:main3}.

\begin{proof}[Proof of Theorem~\ref{thm:main3}]
Assume that the hypotheses of the theorem hold.
Let $\hat{X}_\varepsilon$ be the cotangent lift of $X_\varepsilon$.
By Lemma~\ref{lem:adj} there exists a section $\omega$ of $T^\ast_\Gamma$
 satisfying the AVE \eqref{adjVEconn} of $X^0$ along $\Gamma$
and $\hat{\gamma}(t)=(\gamma(t), \omega_{\gamma(t)})$
is a $T$-periondic solution for $\hat{X}_0$.
Moreover, by Lemma~\ref{lem:vec} $\hat{X}_0$ has a first integral $h_Z$.

Suppose that the system \eqref{sys} has a commutative vector field
$Z_\varepsilon=Z+O(\varepsilon)$ near $\Gamma$.
Then by Lemma~\ref{lem:vec} $h_{Z_\varepsilon}=h_Z+O(\varepsilon)$ is a first integral
of $\hat{X}_\varepsilon$ near $\hat{\Gamma}=\{\hat{\gamma}(t)\mid t\in[0,T)\}$.
Using Lemma~\ref{lem:poisson}, we compute
\begin{align}
\mathscr{I}_{h_Z,\hat{\gamma}}
=&\int_0^{T} dh_Z (\hat{X}^1)({\hat{\gamma}(t)})dt
=\int_0^{T} \{h_{X^1}, h_Z\}({\hat{\gamma}(t)})dt\notag\\
=&\int_0^{T} h_{[X^1, Z]}({\hat{\gamma}(t)})dt
=\int_0^{T} \langle \omega, [X^1, Z]\rangle_{\gamma(t)}  dt\notag\\
=&\int_0^T \omega([X^1, Z])_{\gamma(t)}dt
=\mathscr{J}_{\omega, Z, \gamma}
\label{eqn:thm3}
\end{align}
for $\hat{X}_\varepsilon$.
We apply Theorem~\ref{thm:main1} to complete the proof.
\end{proof}

Theorem~\ref{thm:main3} means that
 if $\mathscr{J}_{\omega,Z,\gamma}\neq0$, 
 then the commutative vector field $Z$ does not persist in \eqref{sys} for $\epsilon>0$.


As in the proof of Theorem~\ref{thm:main3}, we see that
 if $\gamma_\epsilon(t)$ is a $T_\epsilon$-periodic orbit in \eqref{sys},
 then by Lemma~\ref{lem:adj}
 there exists a section $\omega_\epsilon=\omega+O(\epsilon)$ of $T^\ast_{\Gamma_\epsilon}$
 satisfying the AVE \eqref{adjVEconn} of $X_\epsilon$ along $\Gamma_\epsilon$
 and $\hat{\gamma}(t)=(\gamma_\epsilon(t),\omega_{\epsilon,\gamma_\epsilon(t)})$
 is a $T_\epsilon$-periodic orbit for the cotangent lift $\hat{X}_\epsilon$ of $X_\epsilon$,
 where $\Gamma_\epsilon=\{\gamma_\epsilon(t)\mid t\in[0,T_\epsilon)\}$.
Here the section $\omega$ of $T^\ast_\Gamma$
 satisfies the AVE \eqref{adjVEconn} of $X^0$ along $\Gamma$.
Applying Theorem~\ref{thm:main2} to $\hat{X}_\epsilon$ and using \eqref{eqn:thm3},
we obtain the following result on persistence of periodic orbits.

\begin{thm}
\label{thm:main3o}
Assume that (A1) and (A3) hold.
If the perturbed system \eqref{sys} has  a $T_\varepsilon$-periodic orbit
$\gamma_\varepsilon$ depending $C^2$-smoothly on $\varepsilon$
such that $T_0=T$ and $\gamma_0=\gamma$,
then the integral $\mathscr{J}_{\omega, Z,\gamma}$ 
is zero for some section $\omega$ of $T^\ast_\Gamma$
satisfying the AVE \eqref{adjVEconn} of $X^0$ along $\Gamma$.
\end{thm}

Theorem~\ref{thm:main3o} means that
 if $\mathscr{J}_{\omega,Z,\gamma}\neq0$
 for any horizontal section $\omega$ of $\nabla^\ast$, 
 then the periodic orbit $\gamma$ does not persist in \eqref{sys} for $\epsilon>0$.

\subsection{Homoclinic orbits}
We next discuss the persistence of homoclinic orbits and commutative vector fields
 in \eqref{sys}.
Instead of (A3) we assume the following on \eqref{sys0}:
\begin{itemize}
\setlength{\leftskip}{-1.2em}
\item[(A3')]
The unperturbed system \eqref{sys0} has a $C^3$ commutative vector field $Z$
 near $\Gamma^\h$, such that it is linearly independent of $X^0$.
\end{itemize}

In the proof of Lemma~\ref{lem:adj},
we did not essentially use the fact that $\gamma(t)$ is periodic.
So we prove the following lemma similarly.

\begin{lem}
\label{lem:adj2}
Under assumption (A1'),
the connection $\nabla^\ast$ of $T^\ast_{\Gamma^{\h\prime}}$
has a nontrivial horizontal section $\omega^\h:\Gamma^{\h\prime}\to T^\ast_{\Gamma^{\h\prime}}$,
i.e., $\omega^\h$ satisfies the AVE \eqref{adjVEconn} of $X^0$ along $\Gamma^{\h\prime}$
where $\Gamma^{\h\prime}:=\{\gamma^\h(t)\mid t\in \Rset\}$.
\end{lem}

Let $\omega^\h$ be such a horizontal section of $\nabla^\ast$ as stated in Lemma~\ref{lem:adj2},
 and define
\begin{align}\label{Integ-hom-cvf-per}
\tilde{\mathscr{J}}_{\omega^\h, Z,\gamma^\h}
:=\lim_{k\to +\infty}\int_{T_{-k}}^{T_k} \omega^\h([X^1,Z])_{\gamma^{\h}(t)}dt,
\end{align}
where the sequence $\{T_j\}_{j~-\infty}^{\infty}$ is taken as in \eqref{Integ-hom-per}.
If $\gamma^\p$ is an equilibrium, then Eq.~\eqref{Integ-hom-cvf-per} is reduced to
\begin{align}\label{Integ-hom-cvf-eql}
\tilde{\mathscr{J}}_{\omega^\h, Z,\gamma^\h}
 =\int_{-\infty}^{\infty} \omega^\h([X^1,Z])_{\gamma^{\h}(t)}dt
\end{align}
like \eqref{Integ-hom}.

\begin{thm}
\label{thm:main6}
Assume that (A1') and (A3') hold.
If the perturbed system \eqref{sys} has a $C^3$ commutative vector field $Z_\varepsilon$
 depending $C^2$-smoothly on $\varepsilon$ near $\Gamma^\h$ such that $Z_0=Z$, 
 then the limit in the right hand side of \eqref{Integ-hom-cvf-per} exists
 and $\tilde{\mathscr{J}}_{\omega^\h, Z,\gamma}=0$.
\end{thm}

\begin{proof}
If $\gamma^\p$ is a periodic orbit,
 then by Lemma~\ref{lem:adj} there exists a horizontal section $\omega^\p$
 of $T^\ast_{\Gamma^\p}$ satisfying the AVE \eqref{adjVEconn} of $X^0$ along $\Gamma^\p$
 and $(\gamma^\p(t), \omega^\p_{\gamma^\p(t)})$ is a periodic orbit
 for the cotangent lift $\hat{X}^0$ of $X^0$.
Similarly, by assumptions (A1') and Lemma~\ref{lem:adj2},
 we have a homoclinic orbit $({\gamma^\h}(t),\omega^\h_{\gamma^\h(t)})$
 to the periodic orbit $({\gamma^\p}(t), \omega^\p_{\gamma^\p(t)})$ for $\hat{X}^0$,
 where $\omega^\h$ is a horizontal section of $\nabla^\ast$ for $T^\ast_{\Gamma^\h}$.
By applying Theorem~\ref{thm:main4} to the cotangent lift $\hat{X}_\epsilon$ of $X_\epsilon$,
 the rest of the proof is done similarly as in Theorem~\ref{thm:main3}.
\end{proof}

Theorem~\ref{thm:main6} means that
 if $\mathscr{J}_{\omega^\h,Z,\gamma^\h}\neq0$, 
 then the commutative vector field $Z$ does not persist in \eqref{sys} for $\epsilon>0$.

{
\begin{rmk}
Using Theorems~\ref{thm:main1}, \ref{thm:main4}, \ref{thm:main3} and \ref{thm:main6},
 we can determine whether given first integrals and commutative vector fields do not persist
 in \eqref{sys} but  there still exist a sufficient number
 of first integrals and commutative vector fields depending smoothly on the parameter $\epsilon$.
For example, the unperturbed system \eqref{sys0}
 may have different first integrals and commutative vector fields which persist.
So we have to overcome this difficulty
 to extend the results of Poincar\'e \cite{P92} and Kozlov \cite{K83,K96}
 and obtain a sufficient condition for such nonintegrability of the perturbed systems.
\end{rmk}}

As in the proof of Theorem~\ref{thm:main6},
 if $\gamma_\epsilon^\p=\gamma^\p+O(\epsilon)$ is a $T_\epsilon$-periodic orbit
 with $T_\epsilon=T+O(\epsilon)$
 and $\gamma_\epsilon^\h=\gamma^\h+O(\epsilon)$ is a homoclinic orbit
 to $\gamma_\epsilon^\p$ in \eqref{sys},
 then by Lemmas~\ref{lem:adj} and \ref{lem:adj2}
 there exist horizontal sections $\omega_\epsilon^\p=\omega^\p+O(\epsilon)$
 and $\omega_\epsilon^\h=\omega^\h+O(\epsilon)$ of $\nabla^\ast$
 for $\Gamma_\epsilon^\p=\{\gamma_\epsilon^\h(t)\mid t\in[0,T_\epsilon)\}$
 and $\Gamma_\epsilon^{\h\prime}=\{\gamma_\epsilon^\h(t)\mid t\in\Rset\}$,
 respectively,
 so that for the cotangent lift $\hat{X}_\epsilon$ of $X_\epsilon$
 $(\gamma_\epsilon^\p(t),\omega^\p_{\epsilon,\gamma^\p(t)})$ is a periodic orbit
 to which $({\gamma_\epsilon^\h}(t),\omega^\h_{\epsilon,\gamma^\h(t)})$ is a homoclinic orbit.
Here the section $\omega^\h$ of $\nabla^\ast$
 satisfies the AVE \eqref{adjVEconn} along $\Gamma^\h$.
Applying Theorem~\ref{thm:main2} to $\hat{X}_\epsilon$, we obtain the following.

\begin{thm}
\label{thm:main6o}
Assume that (A1') and (A3') hold
 and that there exists a periodic orbit $\gamma_\epsilon^\p$
 depending $C^2$-smoothly on $\varepsilon$
 such that $\gamma^\mathrm{p}_0=\gamma^\mathrm{p}$.
If the perturbed system \eqref{sys} has a homoclinic orbit $\gamma_\epsilon^\h$
 depending $C^2$-smoothly on $\varepsilon$ in \eqref{sys}
 such that $\gamma_0^\h=\gamma^\h$,
 then the limit in the right hand side of \eqref{Integ-hom-cvf-per} exists
 and $\tilde{\mathscr{J}}_{\omega^\h, Z,\gamma}=0$
 for some section $\omega$ of $T^\ast_{\Gamma^\h}$
 satisfying the AVE \eqref{adjVEconn} along $\Gamma^\h$.
\end{thm}

Theorem~\ref{thm:main6o} means that
 if $\tilde{\mathscr{J}}_{\omega^\h,Z,\gamma^\h}\neq0$
 for any horizontal section $\omega^\h$ of $\nabla^\ast$ for $\Gamma^\h$,
 then the homoclinic orbit $\gamma^\h$ does not persists in \eqref{sys} for $\epsilon>0$.


\section{Some relationships with the Melnikov Methods}\label{Melnikov}
In this section, we discuss some relationships of the main results in Sections~2 and 3
 with the {standard, subharmonic and homoclinic Melnikov methods \cite{GH83,M63,W90,Y96},
 which provide sufficient conditions for persistence of periodic and homoclinic orbits, respectively, in time-periodic perturbations of single-degree-of-freedom Hamiltonian systems,
 and with another version of the homoclinic Melnikov method due to Wiggins \cite{W88}
 for autonomous perturbations of multi-degree-of-freedom Hamiltonian systems}.

\subsection{{Standard} Melnikov methods
}\label{subsec-PH}
We first review the {standard Melnikov methods for subharmonic and homoclinic orbits.}
See \cite{GH83,W90,Y96} for more details.

We consider systems of the form
\begin{align}\label{mel}
\dot{x}=J_2DH(x)+\varepsilon g(x,t),\quad
x\in\Rset^2,
\end{align}
where  $\epsilon$ is a small parameter as in the previous sections,
 $H:\mathbb{R}^2\to\mathbb{R}$ and $g:\Rset^2\times\Rset\to\Rset^2$
 are, respectively, $C^3$ and $C^2$ in $x$,
 $g(x,t)$ is $T$-periodic in $t$ with $T>0$ a constant,
 and $J_2$ is the $2\times 2$ symplectic matrix,
\[
J_2=
\begin{pmatrix}
0 & 1\\
-1 & 0
\end{pmatrix}.
\]
When $\epsilon=0$,
 Eq.~\eqref{mel} becomes a single-degree-of-freedom Hamiltonian system
 with the Hamiltonian $H(x)$,
\begin{align}\label{mel0}
\dot{x}=J_2DH(x).
\end{align}
Let $\theta=t\mod T$ so that $\theta\in\Sset_T^1$, where $\Sset_T^1={\Rset/T\Zset}$.
We rewrite \eqref{mel} as an autonomous system,
\begin{equation}\label{mel-aut}
\dot{x}=J_2DH(x)+\varepsilon g(x,\theta),\quad
\dot{\theta}=1.
\end{equation}

We begin with the subharmonic Melnikov method \cite{GH83,W90,Y96},
 and make the following assumption:
\begin{itemize}
\setlength{\leftskip}{-1.5em}
\item[(M)]
The unperturbed system \eqref{mel0}
 possesses a one-parameter family of periodic orbits $q^{\alpha}(t)$
 with period $T^{\alpha}$, $\alpha\in(\alpha_1,\alpha_2)$,  for some $\alpha_1<\alpha_2$.
\end{itemize}
Fix the value of $\alpha\in(\alpha_1,\alpha_2)$ such that
\begin{equation}
lT^\alpha=mT
\label{eqn:res}
\end{equation}
for some relatively prime integers $l,m>0$.
When $\epsilon=0$, Eq.~\eqref{mel-aut}
 has a one-parameter family of $mT$-periodic orbits
 $(x,\theta)=(q^\alpha(t-\tau),t)$, $\tau\in[0,T)$.
Note that $(x,\theta)=(q^\alpha(t-\tau-jT),t)$
 represents the same periodic orbit in the phase space $\Rset^2\times\Sset_T^1$
 for $j=0,1,\ldots,m-1$.
Define the \emph{subharmonic Melnikov function} as
\begin{equation}
M^{m/l}(\tau):=\int_{0}^{mT} DH(q^{\alpha}(t))\cdot g(q^{\alpha}(t),t+\tau)dt,
\label{eqn:subM}
\end{equation}
where the dot `$\cdot$' represents the standard inner product in $\Rset^2$.
We have the following (see \cite{GH83,W90,Y96} for the proof).

\begin{thm}
\label{thm:M1}
If the subharmonic Melnikov function $M^{l/m}(\tau)$ has a simple zero
 at $\tau=\tau_0\in\Sset_T^1$,
 then for $\epsilon>0$ sufficiently small
 Eq.~\eqref{mel-aut} has a periodic orbit of period $mT$
 near the unperturbed periodic orbit $(x,\theta)=(q^\alpha(t-\tau_0),t)$
 satisfying \eqref{eqn:res}.
\end{thm}

Theorem~\ref{thm:M1} means that 
 the periodic orbit $(x,\theta)=(q^\alpha(t-\tau_0),t)$ persists
 in \eqref{mel-aut} for $\epsilon>0$ sufficiently small
 if $M^{m/l}(\tau)$ has a simple zero at $\tau=\tau_0$.
The stability of the perturbed periodic orbit can be also determined easily \cite{Y96}.
Moreover, several bifurcations of the periodic orbits
 were discussed in \cite{Y96,Y02,Y03}.


We next review the homoclinic Melnikov method \cite{GH83,M63,W90}
 and assume the following instead of (M):
\begin{itemize}
\setlength{\leftskip}{-1.2em}
\item[(M')]
The unperturbed system \eqref{mel0} possesses a hyperbolic  saddle point $p$
 connected to itself by a homoclinic orbit $q^\h(t)$.
\end{itemize}
When $\epsilon=0$, 
 Eq.~\eqref{mel-aut} has a hyperbolic $T$-periodic orbit $(x,\theta)=(p,t)$
 with a one-parameter family of homoclinic orbits
 $(x,\theta)=(q^\h(t-\tau),t)$, $\tau\in\Sset_T^1$.
Note that $(x,\theta)=(q^\h(t-\tau-jT),t)$
 represents the same homoclinic orbit in the phase space $\Rset^2\times\Sset_T^1$
 for $j=0,1,\ldots,m-1$.
We easily show that there exists a hyperbolic periodic orbit near $(x,\theta)=(p,t)$
 (see \cite{GH83,W90} for the proof).
Define the \emph{homoclinic Melnikov function} as 
\begin{equation}
M(\tau):=\int_{-\infty}^{\infty} DH(q^{\h}(t))\cdot g(q^{\h}(t),t+\tau)dt
\label{eqn:homM}
\end{equation}
We have the following (see \cite{GH83,M63,W90} for the  proof).

\begin{thm}
\label{thm:M2}
If the homoclinic Melnikov function $M(\tau)$ has a simple zero at $\tau=\tau_0\in\Sset_T^1$,
 then for $\epsilon>0$ sufficiently small
 Eq.~\eqref{mel-aut} has a transverse homoclinic orbit to the hyperbolic periodic orbit
 near $(x,\theta)=(q^\h(t-\tau_0),t)$.
\end{thm}

Theorem~\ref{thm:M2} means that 
 the homoclinic orbit $(x,\theta)=(q^\h(t-\tau_0),t)$ persists
 in \eqref{mel-aut} for $\epsilon>0$ sufficiently small
 if $M(\tau)$ has a simple zero at $\tau=\tau_0$.
By the Smale-Birkhoff theorem \cite{GH83,W90},
 the existence of transverse homoclinic orbits to hyperbolic periodic orbits
 implies that chaotic motions occur in \eqref{mel-aut}, i.e., in \eqref{mel}.
 
{We now describe some relationships of our results
 on persistence of first integrals
 with the standard Melnikov methods for \eqref{mel-aut},
 which has the Hamiltonian $H(x)$ is a first integral when $\epsilon=0$.}
We first state the relationship for the subharmonic Melnikov method.

\begin{thm}\label{thm:sub-mel}
Suppose that assumption (M) and the resonance condition $lT^\alpha=mT$ hold
 for $l,m>0$ relatively prime integers.
If Eq.~\eqref{mel-aut} has a $C^3$ first integral $F_\varepsilon(x,t)=H(x)+O(\epsilon)$ 
 depending $C^2$-smoothly on $\varepsilon$ in a neighborhood of
\[
\Gamma_{\tau_0}^\alpha=\{(q^\alpha(t-\tau_0),t)\mid t\in[0,mT)\}
\]
with $\tau_0\in\Sset_T^1$, then there exists a connected open set ${\Pi}\subset\Sset_T^1$
 such that $\tau_0\in {\Pi}$
 and the subharmonic Melnikov function $M^{m/l}(\tau)$ is zero on ${\Pi}$.
\end{thm}

\begin{proof}
Assume that the hypotheses of the theorem hold
 and $F_\epsilon$ is a first integral of \eqref{mel-aut}.
Then $\hat{\gamma}^{m/l}_\tau(t)=(q^\alpha(t-\tau), t)$
 is an $mT$-periodic orbit in \eqref{mel-aut} with $\epsilon=0$ for any $\tau\in [0, T)$.
Letting $F=H$, we write the integral \eqref{integral1} as
\begin{align*}
\mathscr{I}_{H,\hat{\gamma}^{m/l}_\tau}
=&\int_0^{mT}DH(q^\alpha(t-\tau))\cdot g(q^\alpha(t-\tau),t)dt\\
=&\int_0^{mT}DH(q^\alpha(t))\cdot g(q^\alpha(t),t+\tau)dt,
\end{align*}
which coincides with $M^{m/l}(\tau)$.
We choose a connected open set ${\Pi}\subset\Sset_T^1$
 such that the neighborhood of  $\Gamma_{\tau_0}^\alpha$
 contains $\bigcup_{\tau\in {\Pi}}\Gamma_\tau^\alpha$.
Applying Theorem~\ref{thm:main1} to the unperturbed periodic orbit
 $\hat{\gamma}^{m/l}_\tau$ for $\tau\in {\Pi}$,
 we obtain the desired result.
\end{proof}

Theorem~\ref{thm:sub-mel} means that
 if there exists a connected open set ${\Pi\subset\Sset_T^1}$
 such that $M^{m/l}(\tau)\not\equiv 0$ on ${\Pi}$,
 then the first integral $H$ does not persist near $\bigcup_{\tau\in\Pi}\Gamma_\tau^\alpha$
 in \eqref{mel-aut} for $\epsilon>0$.

\begin{rmk}
\label{rmk:4.3a}
Under the hypotheses of Theorem~\ref{thm:sub-mel} the following hold:
\begin{enumerate}
\setlength{\leftskip}{-1.5em}
\item[(i)]
It follows from Theorem~\ref{thm:main2} that
 if the periodic orbit $(x,\theta)=(q^\alpha(t-\tau),t)$ persists in \eqref{mel-aut},
 then $M^{m/l}(\tau)=0$;
\item[(ii)]
If Eq.~\eqref{mel-aut} has such a first integral near $\bigcup_{\tau\in\Sset_T^1}\Gamma_\tau^\alpha$,
 then $M^{m/l}(\tau)$ is identically zero on $\Sset_T^1$;
\item[(iii)]
If $H,g$ are analytic and Eq.~\eqref{mel-aut}
 has such a first integral near $\Gamma_{\tau_0}^\alpha$ with {some} $\tau_0\in{\Sset_T^1}$,
 then $M^{m/l}(\tau)$ is identically zero on $\Sset_T^1$.
\end{enumerate}
The statement of part~(i) consists with Theorem~\ref{thm:M1}.
Part~(iii) follows from the identity theorem (e.g., Theorem~3.2.6 of \cite{AF03})
 since $M^{m/l}(\tau)$ is also analytic.
\end{rmk}

Similarly, we have the following result for the homoclinic Melnikov method.

\begin{thm}\label{thm:hom-mel}
Suppose that assumption (M') holds.
If Eq.~\eqref{mel-aut} has a $C^3$ first integral $F_\epsilon{(x,t)}=H(x)+O(\epsilon)$ 
 depending $C^2$-smoothly on $\varepsilon$ in a neighborhood of
\[
\Gamma_{\tau_0}^\h=\{(q^\h(t-\tau_0),t)\mid t\in\Rset\}
\]
with $\tau_0\in\Sset_T^1$, then there exists a connected open set ${\Pi}\subset\Sset_T^1$
 such that $\tau_0\in {\Pi}$
 and the homoclinic Melnikov function $M(\tau)$ is zero on ${\Pi}$.
\end{thm}

\begin{proof}
Assume that (M') holds.
Then  in \eqref{mel-aut} with $\epsilon=0$,
 $(p, {t})$ represents a periodic orbit,
 to which $\hat{\gamma}_\tau^\h(t)=(q^\h({t-\tau), t})$ is a homoclinic orbit, for any $\tau\in[0, T)$.
We take {the Poincar\'e section} $\Sigma=\{(x,\theta)\in\Rset^2\times\Sset_T^1\mid\theta=0\}$
 and set $T_j= jT$, $j\in\Zset$.
Letting $F=H$, we write the integral in \eqref{Integ-hom-per} as
\[
\int_{-jT}^{jT} DH(q^\h(t-\tau))\cdot g(q^\h(t-\tau), t)dt
=\int_{-jT}^{jT} DH(q^\h(t))\cdot g(q^\h(t), t+\tau)dt,
\]
which converges to $M(\tau)$ {as $j\to\infty$}.
We choose a connected open set ${\Pi}\subset\Sset_T^1$
 such that the neighborhood of  $\Gamma_{\tau_0}^\h$
 contains $\bigcup_{\tau\in {\Pi}}\Gamma_\tau^\h$.
Applying Theorem~\ref{thm:main4} to the unperturbed homoclinic orbit
 $\hat{\gamma}_\tau^\h$ for $\tau\in{\Pi}$,
 we obtain the desired result.
\end{proof}

Theorem~\ref{thm:hom-mel} means that
 if there exists a connected open set ${\Pi}\in\Sset_T^1$
 such that $M(\tau)\not\equiv 0$ on ${\Pi}$,
 then the first integral $H$ does not persist near $\bigcup_{\tau\in {\Pi}}\Gamma_\tau^\h$
 in \eqref{mel-aut} for $\epsilon>0$.

\begin{rmk}
\label{rmk:4.3b}
Under the hypotheses of Theorem~\ref{thm:hom-mel} the following hold,
 as in Remark~\ref{rmk:4.3a}:
\begin{enumerate}
\setlength{\leftskip}{-1.5em}
\item[(i)]
It follows from Theorem~\ref{thm:main5} that
 if the homoclinic orbit $(x,\theta)=(q^\h(t-\tau),t)$ persists in \eqref{mel-aut},
 then $M(\tau)=0$;
\item[(ii)]
If Eq.~\eqref{mel-aut} has such a first integral near $\bigcup_{\tau\in\Sset_T^1}\Gamma_\tau^\h$,
 then $M(\tau)$ is identically zero on $\Sset_T^1$;
\item[(iii)]
If $H,g$ are analytic and Eq.~\eqref{mel-aut}
 has such a first integral near $\Gamma_{\tau_0}^\h$ with {some} $\tau_0\in\Sset_T$,
 then $M(\tau)$ is identically zero on $\Sset_T^1$.
\end{enumerate}
The statement of part~(i) consists with Theorem~\ref{thm:M2}.
\end{rmk}

{
\subsection{Another version of the homoclinic Melnikov method}

We next consider $(m+1)$-degree-of-freedom Hamiltonian systems of the form
\begin{equation}
\begin{split}
&
\dot{x}=J_{2m}D_xH^0(x,I)+\epsilon J_{2m}D_xH^1(x,I,\theta),\\
&
\dot{I}=-\epsilon D_\theta H^1(x,I,\theta),\\
&
\dot{\theta}=D_IH^0(x,I)+\epsilon D_IH^1(x,I,\theta),
\end{split}\quad
(x,I,\theta)\in\Rset^{2m}\times V\times\Sset_{2\pi}^1,
\label{g-mel}
\end{equation}
for which $H_\epsilon(x,I,\theta)=H^0(x,I)+\epsilon H^1(x,I,\theta)$ is the Hamiltonian,
 where $m\ge 1$ is an integer, $V\subset\Rset$ is an open interval,
 $H^0(x,I),H^1(x,I,\theta)$ are $C^3$ in $(x,I,\theta)$,
 and $J_{2m}$ is the $2m\times 2m$ symplectic matrix given by
\[
J_{2m}=
\begin{pmatrix}
0 & \id_m\\
-\id_m & 0
\end{pmatrix},
\]
where $\id_m$ is the $m\times m$ identity matrix.
When $\epsilon=0$, Eq.~\eqref{g-mel} becomes
\begin{align}\label{g-mel0}
\dot{x}=J_{2m}D_xH^0(x,I),\quad
\dot{I}=0,\quad
\dot{\theta}=D_IH^0(x,I).
\end{align}
{Note that $I$ and $\theta$ are scalar variables.}
We assume the following on the unperturbed system \eqref{g-mel0}:

\begin{enumerate}
\setlength{\leftskip}{-1em}
\item[(W1)]
For each $I\in V$, the first equation is 
 has $m$ $C^3$ first integrals $F_j(x,I)$, $j=1,\ldots,m$, with $F_1(x,I)=H^0(x,I)$
 such that $D_xF_j(x,I)$, $j=1,\ldots,m$, are linearly independent except at equilibria
 {and they are in involution, i.e.,
 $\{F_i(x,I), F_j(x,I)\}:=D_xF_i(x,I)\cdot J_{2m}D_xF_j(x,I)=0, \ i,j=1, \ldots,m$.} 
\item[(W2)]
For each $I\in V$
 the first equation has a hyperbolic equilibrium $x^I$
 and an $(m-1)$-parameter family of homoclinic orbits $q^I(t;\alpha)$,
 $\alpha\in \bar{V}\subset\Rset^{m-1}$,
 to $x^I$,
 where $x^I$ and $q^I(t;\alpha)$ depend $C^2$-smoothly on $I$ and $\alpha$,
 and $\bar{V}$ is an connected open in $\Rset^{m-1}$.
\item[(W3)]
$D_I H^0(q^I(t;\alpha),I)>0$ for $(I,\alpha)\in V\times \bar{V}$.
\end{enumerate}
Obviously, $I$ is a first integral of \eqref{g-mel0} as well as $F_j(x,I)$, $j=1,\ldots,m$,
 so that the Hamiltonian system \eqref{g-mel} is Liouville integrable \cite{A89,M99}.
Thus, Eq.~\eqref{g-mel} is a special case in a class of systems called ``System III'' in Chapter~4 of \cite{W88},
 in which very wide classes of systems containing more general Hamiltonian systems{,
 especially having multiple action and angular variables
 such as the scalar variables $I$ and $\theta$ in \eqref{g-mel},} were discussed.

In \eqref{g-mel0} $\N_0=\{(x^I,I,\theta)\mid I\in V,\theta\in\Sset_{2\pi}^1\}$
 is a {two-dimensional} normally hyperbolic invariant manifold {with boundary}
 whose stable and unstable manifolds coincide along the homoclinic manifold
\[
\bar{\Gamma}^\h
=\{(q^I(t;\alpha),I,\theta)\mid I\in V,\alpha\in \bar{V},\theta\in\Sset_{2\pi}^1\}.
\]
Here ``normal hyperbolicity'' means that
 the expansive and contraction rates of the flow generated by \eqref{g-mel0}
 normal to $\N_0$ dominate those tangent to $\N_0$.
{Note that $(x,I,\theta)=(x^{I_0},I_0,D_IH^0(x^{I_0},I_0)t+\theta_0)$ represents a periodic orbit on $\N_0$
 for $(I_0,\theta_0)\in V\times\Sset_{2\pi}^1$.}
{Using} the invariant manifold theory \cite{W94}{, we show} that when $\epsilon\neq 0$
 Eq.~\eqref{g-mel} also has a {two-dimensional} normally hyperbolic invariant manifold $\N_\epsilon$
  near $\N_0$
 and its stable and unstable manifolds are close to those of $\N_0$.
Moreover, the invariant manifold $\N_\epsilon$
 consists of periodic orbits $\gamma_{I,\epsilon}^\p${,
 which are given as intersections between $\N_\epsilon$
 and the level sets $H_\epsilon(x,I,\theta)=\mathrm{const.}$ since $\D_IH^0(x^I,I)>0$ by (W3),}
 near $\gamma_I^\p=\{(x^I,I,\theta)|\theta\in\Sset_{2\pi}^1\}$ for $I\in V$.
{Note that $\N_\epsilon$ can be invariant by taking two periodic orbits as its boundary
 as in Proposition~2.1 of \cite{Y00}.}

Let $\theta=\theta^I(t;\alpha)$ denote the solution to
\[
\dot{\theta}=D_IH^0(q^I(t;\alpha),I)
\]
with $\theta(0)=0$, i.e.,
\[
\theta^I(t;\alpha)=\int_0^t D_IH^0(q^I(t;\alpha),I)dt.
\]
Then $\gamma_{I,\alpha,\theta_0}^\h(t)=(q^I(t;\alpha),I,\theta^I(t;\alpha)+\theta_0)$
  is a homoclinic orbit to the periodic orbit $\gamma_I^\p$ in \eqref{g-mel0}
  for any $\theta_0\in\Sset_{2\pi}^1$.
Let $\{T_j^{I,\alpha}\}_{j=-\infty}^\infty$ be a sequence for $(I,\alpha)\in V\times\bar{V}$
 such that
\begin{equation}
\theta^I(T_j^{I,\alpha};\alpha)=0,\quad
j\in\Zset,\quad\mbox{and}\quad
\lim_{j\to\pm\infty}T_j^{I,\alpha}=\pm\infty.
\label{eqn:Tj}
\end{equation}
By assumption~(W3) there exists such an sequence {$\{T_j^{I,\alpha}\}_{j=-\infty}^\infty$}.
Define the \emph{Melnikov functions} for \eqref{g-mel} as
\begin{equation}
\bar{M}_1^I(\theta_0,\alpha)
=\lim_{j\to\infty}\int_{T_{-j}^{I,\alpha}}^{T_j^{I,\alpha}}
 D_\theta H^1(q^I(t;\alpha),I,\theta^I(t;\alpha)+\theta_0)dt
\label{eqn:barM1}
\end{equation}
and
\begin{align}
&
\bar{M}_k^I(\theta_0,\alpha)\notag\\
&
=\lim_{j\to\infty}\int_{T_{-j}^{I,\alpha}}^{T_j^{I,\alpha}}\bigl(D_x F_k(q^I(t;\alpha),I)\cdot
  J_{2m}D_xH^1(q^I(t;\alpha),I,\theta^I(t;\alpha)+\theta_0)\notag\\
&\qquad\qquad
 -D_I F_k(q^I(t;\alpha),I)D_\theta H^1(q^I(t;\alpha),I,\theta^I(t;\alpha)+\theta_0)\bigr)dt
\label{eqn:barMk}
\end{align}
for $k=2,\ldots,m$.
Note that the definitions of the Melnikov functions $\bar{M}_k^I$, $k\ge 2$,
 are different from the original ones of \cite{W88}.
We call $\bar{M}^I=(\bar{M}_1^I,\ldots,\bar{M}_m^I)$ the \emph{Melnikov vector}.
From Theorem~4.1.19 of \cite{W88}
 we obtain the following result for \eqref{g-mel}.

\begin{thm}\label{thm:W}
Suppose that assumptions (W1)-(W3) hold.
If for some $I\in V$
\begin{enumerate}
\setlength{\leftskip}{-1.8em}
\item[(i)]
$\bar{M}^I(\theta,\alpha)=0$;
\item[(ii)]
$\det D\bar{M}^I(\theta,\alpha)\neq 0$
\end{enumerate}
at $(\theta,\alpha)=(\theta_0,\alpha_0)$,
 then the $(m+1)$-dimensional stable and unstable manifolds
 $W^\s(\gamma_{I,\epsilon}^\p)$ and $W^\u(\gamma_{I,\epsilon}^\p)$
 intersect transversely near $(x,I,\theta)=(q^I(0;\alpha_0),I,\theta_0)$
 on the level set of $H_\epsilon(\gamma_{I,\epsilon}^\p)$.
\end{thm}

\begin{proof}
Assume that the hypotheses of Theorem~\ref{thm:W} hold.
Let $\tilde{M}_1^I(\theta,\alpha)=\bar{M}_1^I(\theta,\alpha)$ and
\[
\tilde{M}_k^I(\theta,\alpha)
 =\bar{M}_k^I(\theta,\alpha)+D_IF_k(x_I,I)\bar{M}_1^I(\theta,\alpha),\quad
 k=2,\ldots,m,
\]
and let $\tilde{M}^I=(\tilde{M}_1^I,\ldots,\tilde{M}_m^I)$,
 which is the original Melnikov vector defined in \cite{W88} for \eqref{g-mel}.
Note that in \cite{W88},
 although a time sequence does not appear in its formulas (4.1.84) and (4.1.85) or (4.1.101) and (4.1.102),
 such conditional convergences of the integrals as \eqref{eqn:barM1} and \eqref{eqn:barMk}
 are implicitly assumed (see his arguments on system III in part iii) of Section~4.1d of \cite{W88}).
We see that if $\bar{M}^I(\theta,\alpha)$ satisfies conditions~(i) and (ii)
 at $(\theta,\alpha)=(\theta_0,\alpha_0)$,
 then $\tilde{M}^I(\theta_0,\alpha_0)=0$ and $\det D\tilde{M}^I(\theta_0,\alpha_0)\neq 0$,
since $\det D\tilde{M}^I(\theta_0,\alpha_0)=\det D\bar{M}^I(\theta_0,\alpha_0)$,
 which follows from
\begin{align*}
D\tilde{M}^I(\theta_0,\alpha_0)
=&
\begin{pmatrix}
D_\theta\bar{M}_1^I(\theta_0,\alpha_0),D_\alpha\bar{M}_1^I(\theta_0,\alpha_0)\\
D_\theta\bar{M}_2^I(\theta_0,\alpha_0),D_\alpha\bar{M}_2^I(\theta_0,\alpha_0)\\
\vdots\\
D_\theta\bar{M}_m^I(\theta_0,\alpha_0),D_\alpha\bar{M}_m^I(\theta_0,\alpha_0)
\end{pmatrix}\\
& +
\begin{pmatrix}
0\\
D_IF_2(x_I,I)(D_\theta\bar{M}_1^I(\theta_0,\alpha_0),
 D_\alpha\bar{M}_1^I(\theta_0,\alpha_0))\\
\vdots\\
D_IF_m(x_I,I)(D_\theta\bar{M}_1^I(\theta_0,\alpha_0),
 D_\alpha\bar{M}_1^I(\theta_0,\alpha_0))
\end{pmatrix}.
\end{align*}
We obtain the desired result from Theorem~4.1.19 of \cite{W88}.
\end{proof}

\begin{rmk}\label{rmk:4.2}
The Melnikov vector $\bar{M}^I(\theta,\alpha)$
 does not depend on the choice of time sequence $\{T_j^{I,\alpha}\}_{j=-\infty}^\infty$.
Actually, letting $\{\hat{T}_j^{I,\alpha}\}_{j=-\infty}^\infty$ be a different time sequence satisfying
\[
\theta^I(\hat{T}_j^{I,\alpha};\alpha)=\hat{\theta}_0,\quad
j\in\Zset,\quad\mbox{and}\quad
\lim_{j\to\pm\infty}\hat{T}_j^{I,\alpha}=\pm\infty
\]
instead of \eqref{eqn:Tj}, we have
\begin{align*}
\hat{M}_1^I(\theta,\alpha)
:=&\lim_{j\to\infty}\int_{\hat{T}_{-j}^{I,\alpha}}^{\hat{T}_j^{I,\alpha}}
 D_\theta H^1(q^I(t;\alpha),I,\theta^I(t;\alpha)+\theta)dt\\
=&\bar{M}_1^I(\theta,\alpha)
 +\lim_{j\to\infty}\biggl(\int_{\hat{T}_{-j}^{I,\alpha}}^{T_{-j}^{I,\alpha}}
 +\int_{T_{j}^{I,\alpha}}^{\hat{T}_{j}^{I,\alpha}}\biggr)
 D_\theta H^1(q^I(t;\alpha),I,\theta^I(t;\alpha)+\theta)dt\\
=&\bar{M}_1^I(\theta,\alpha)
\end{align*}
since
\begin{align*}
&
\lim_{t\to\pm\infty}D_\theta H^1(q^I(t;\alpha),I,\theta)
=D_\theta H^1(x^I,I,\theta),\\
&
\lim_{t\to\pm\infty}D_I H^0(q^I(t;\alpha),I)
=D_I H^0(x^I,I)
\end{align*}
and
\begin{align*}
&
\biggl(\int_{\hat{T}_{-j}^{I,\alpha}}^{T_{-j}^{I,\alpha}}
 +\int_{T_{j}^{I,\alpha}}^{\hat{T}_{j}^{I,\alpha}}\biggr) D_\theta H^1(x^I,I,\theta^I(t;\alpha)+\theta)
 D_IH^0(q^I(t;\alpha),I)
 dt\\
&=\biggl(\int_{\hat{T}_{-j}^{I,\alpha}}^{T_{-j}^{I,\alpha}}
 +\int_{T_{j}^{I,\alpha}}^{\hat{T}_{j}^{I,\alpha}}\biggr)
 \frac{d}{dt}
H^1(x^I,I,\theta^I(t;\alpha)+\theta)\, dt
=0.
\end{align*}
Similarly, we show
\begin{align*}
\hat{M}_k^I(\theta,\alpha)
:=&\lim_{j\to\infty}\int_{\hat{T}_{-j}^{I,\alpha}}^{\hat{T}_j^{I,\alpha}}\bigl(D_x F_k(q^I(t;\alpha),I)\cdot
  J_{2m}D_xH^1(q^I(t;\alpha),I,\theta^I(t;\alpha)+\theta)\notag\\
&\qquad\qquad
 -D_I F_k(q^I(t;\alpha),I)D_\theta H^1(q^I(t;\alpha),I,\theta^I(t;\alpha)+\theta)\bigr)dt\\
 =& \bar{M}_k^I(\theta,\alpha),\quad
k=2,\ldots,m.
\end{align*}
 \end{rmk}

Theorem~\ref{thm:W} means that the homoclinic orbit
 $\gamma_{I,\alpha,\theta_0}^\h(t)$ persists in \eqref{g-mel}
 for $\epsilon>0$ sufficiently small
 if the Melnikov vector $\bar{M}^I(\theta,\alpha)$ satisfies its hypotheses.
By the Smale-Birkhoff theorem \cite{GH83,W90}, such transverse intersection
 between the stable and unstable manifolds of periodic orbits
 implies that chaotic motions occur in \eqref{g-mel}.

We now describe a relationship of our results on persistence of first integrals
 with the homoclinic Melnikov methods for \eqref{g-mel},
 in which the Hamiltonian $H_\epsilon(x,I,\theta)$ is always a persisting first integral.
We have the following result.

\begin{thm}\label{thm:hom-mel2}
Suppose that assumptions (W1)-(W3) hold.
If the Hamiltonian system~\eqref{g-mel}
 has a $C^3$ first integral $F_{k,\epsilon}(x,I,\theta)=F_k(x,I)+O(\epsilon)$
 {\rm(}resp. $F_{m+1,\epsilon}(x,I,\theta)=I+O(\epsilon))$
 depending $C^2$ smoothly on $\epsilon$
 in a neighborhood of
\[
\bar{\Gamma}_{I_0,\theta_0,\alpha_0}^\h
 =\{(q^{I_0}(t;\alpha_0),I_0,\theta^{I_0}(t;\alpha_0))\mid t\in\Rset\}
\]
for some $k=2,\ldots,m$,
 then there exists a connected open set
 $\bar{\Pi}\subset V\times\Sset_{2\pi}^1\times\bar{V}$
 such that $(I_0,\theta_0,\alpha_0)\in\bar{\Pi}$
 and the Melnikov function $\bar{M}_k^{I}(\theta,\alpha)=0$
 {\rm(}resp. $\bar{M}_1^{I}(\theta,\alpha)=0)$ on $\bar{\Pi}$.
\end{thm}

\begin{proof}
Assume that (W1)-(W3) hold.
We choose the Poincar\'e section
 $\Sigma=\{(x,I,\theta)\in\Rset^{2m}\times V\times\Sset_{2\pi}^1\mid\theta=\theta_0\}$
 and take $T_j= T_j^{I,\alpha}$, $j\in\Zset$ (cf. Eq.~\eqref{eqn:Tj}).
Letting $F=F_k$ for $k=2,\ldots,m$ (resp. $F=I$),
 we write the integral in \eqref{Integ-hom-per} as
\begin{align*}
&
\int_{T_{-j}^{I,\alpha}}^{T_j^{I,\alpha}}\bigl(D_x F_k(q^I(t;\alpha),I)\cdot
 J_{2m}D_x H^1(q^I(t;\alpha),I,\theta^I(t;\alpha)+\theta_0)\\
&\qquad\quad
-D_I F_k(q^I(t;\alpha),I)
D_\theta H^1(q^I(t;\alpha),I,\theta^I(t;\alpha)+\theta_0)\bigr)dt\\[1ex]
&\qquad
\left(\mbox{resp.}\quad
\int_{T_{-j}^{I,\alpha}}^{T_j^{I,\alpha}}
 D_\theta H^1(q^I(t;\alpha),I,\theta^I(t;\alpha)+\theta_0)dt
 \right)
\end{align*}
for the homoclinic orbit $\gamma_{I,\theta,\alpha_0}^\h(t)$.
We choose a connected open set $\bar{\Pi}\subset V\times\Sset_{2\pi}^1\times\bar{V}$
 such that the neighborhood of  $\bar{\Gamma}_{I_0,\theta_0,\alpha_0}^\h$
 contains $\bigcup_{(I,\theta,\alpha)\in\bar{\Pi}}\bar{\Gamma}_{I,\theta,\alpha}^\h$.
Applying Theorem~\ref{thm:main4} to the unperturbed homoclinic orbit
 $\gamma_{I,\theta,\alpha}^\h(t)$ for $(I,\theta,\alpha)\in\bar{\Pi}$,
 we obtain the desired result.
\end{proof}
 \begin{rmk}
Under the hypotheses of Theorem~\ref{thm:hom-mel2}
 the following hold  as in Remarks~\ref{rmk:4.3a} and \ref{rmk:4.3b}:
\begin{enumerate}
\setlength{\leftskip}{-1.5em}
\item[(i)]
It follows from Theorem~\ref{thm:main5} that
 if the homoclinic orbit $\gamma_{I_0,\theta_0,\alpha_0}^\h(t)$ persists in \eqref{g-mel},
 then $\bar{M}^{I_0}(\theta_0,\alpha_0)=0$;
\item[(ii)]
If Eq.~\eqref{g-mel} has such a first integral near $\bar{\Gamma}^\h$,
 then the corresponding Melnikov function
 is identically zero on $V\times\Sset_{2\pi}^1\times\bar{V}$;
\item[(iii)]
If $H^0,H^1$ are analytic
 and Eq.~\eqref{g-mel} has such a first integral except for $H_\epsilon$
 near $\bar{\Gamma}_{I_0,\theta_0,\alpha_0}^\h$
 with some $(I_0,\theta_0,\alpha_0)\in\bar{\Pi}$,
 then $\bar{M}^{I}(\theta)=0$ is identically zero on $\tilde{\Gamma}^\h$.
\end{enumerate}
The statement of part (i) consists with Theorem~\ref{thm:W}.
\end{rmk}
}

\section{Examples}\label{Application}

We now illustrate the above theory for {four} examples:
 The periodically forced Duffing oscillator \cite{GH83,W90,Y94,Y96},
 {two identical pendula coupled with a harmonic oscillator,}
 a periodically forced rigid body \cite{Y18}
 and a three-mode truncation of a buckled beam \cite{Y01}.

\subsection{Periodically forced Duffing oscillator}

\begin{figure}[t]
\includegraphics[scale=0.27,bb=0 0 690 572]{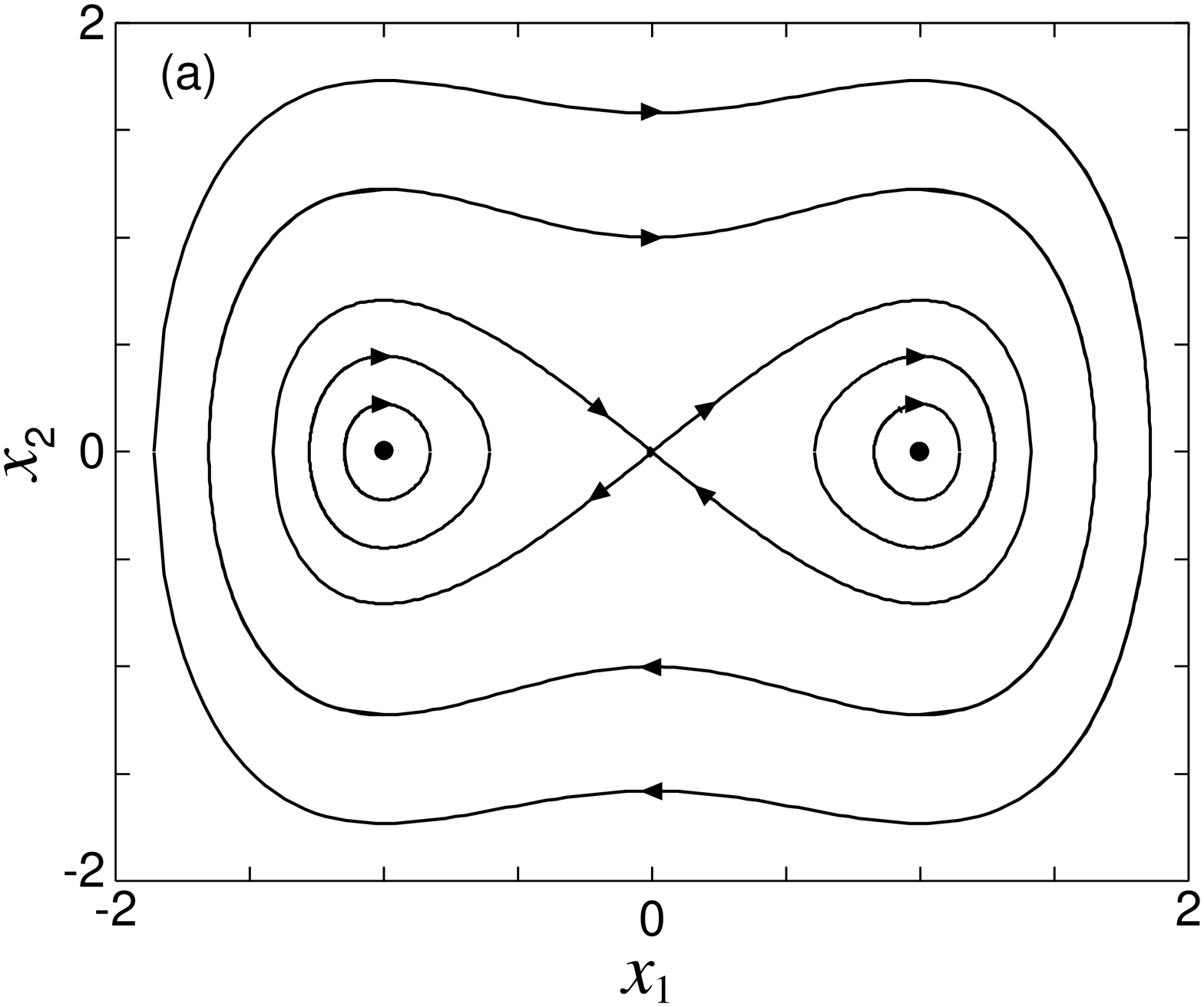}\quad
\includegraphics[scale=0.28,bb=0 0 559 551]{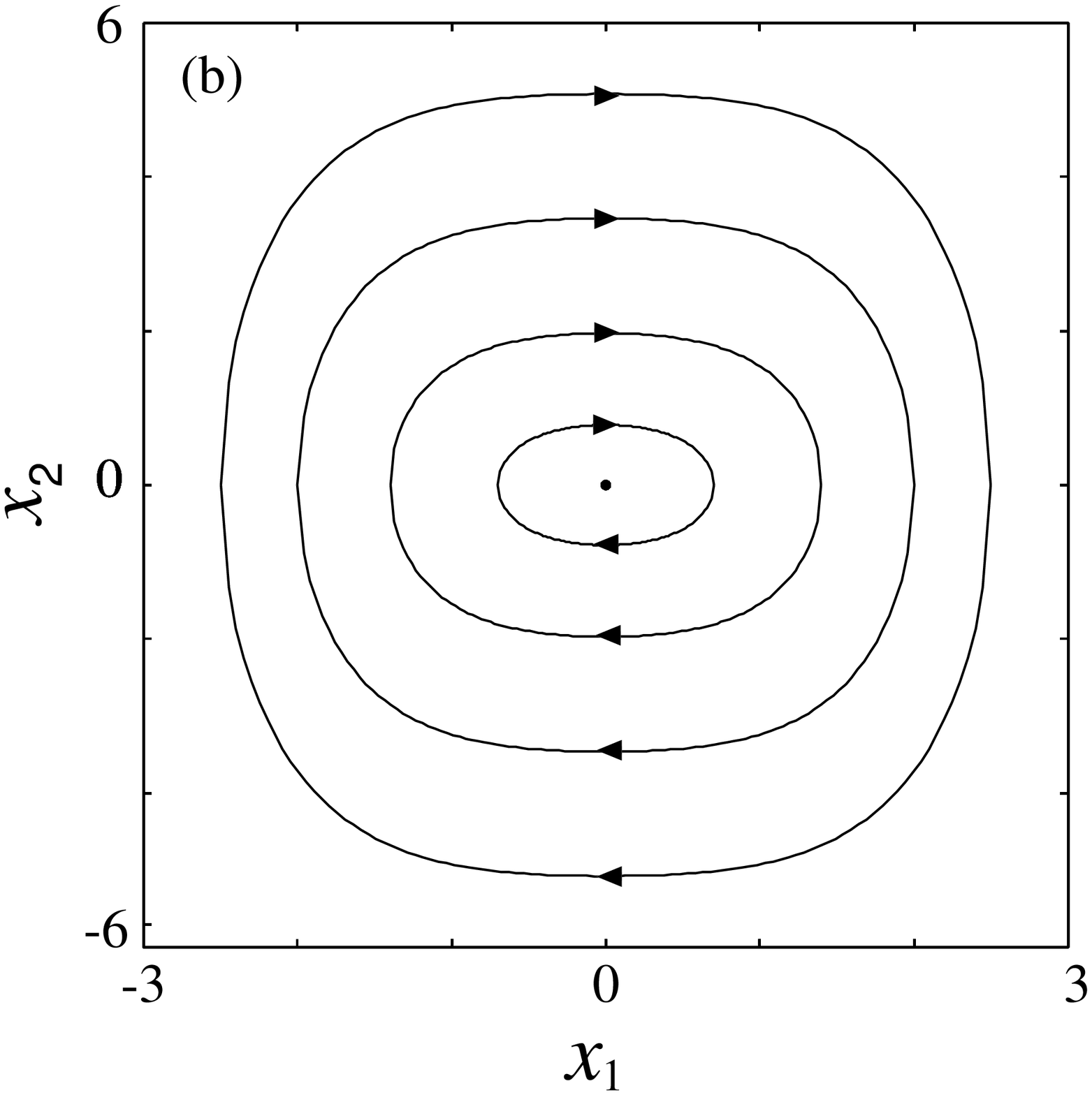}
\caption{Phase portraits of \eqref{Duffing} with $\epsilon=0$: (a) $a=1$; (b) $a=-1$.
\label{fig:5a}}
\end{figure}

We {first} consider the periodically forced Duffing oscillator
\begin{equation}\label{Duffing}
\dot{x}_1=x_2,\quad
\dot{x}_2=ax_1-x_1^3+\epsilon(\beta\cos\omega t-\delta x_2),
\end{equation}
where {$x_1,x_2\in\Rset$},
 $a=1$ or $-1$, and $\beta,\delta,\omega$ are positive constants.
When $\epsilon=0$,
 Eq.~\eqref{Duffing} becomes a single-degree-of-freedom Hamiltonian system
 with the Hamiltonian
\begin{equation}\label{Ham-Duffing}
H=-\frac{1}{2}ax_1^2+\frac{1}{4}x_1^4+\frac{1}{2}x_2^2,
\end{equation}
and it is a special case of \eqref{mel-aut}.
See Fig.~\ref{fig:5a} for the phase portraits of \eqref{Duffing} with $\epsilon=0$.

We begin with the case of $a=1$.
When $\epsilon=0$, in the phase plane
 there exist a pair of homoclinic orbits
\[
q^\h_\pm(t)=(\pm\sqrt{2}\sech t, \mp\sqrt{2}\sech t\,\tanh t),
\]
a pair of one-parameter families of periodic orbits
\begin{align*}
q^k_{\pm}(t)
 =&\biggl(\pm\frac{\sqrt{2}}{\sqrt{2-k^2}}\dn\left(\frac{t}{\sqrt{2-k^2}}\right),\\
& \quad\mp\frac{\sqrt{2}k^2}{2-k^2}\sn\left(\frac{t}{\sqrt{2-k^2}}\right)
\cn\left(\frac{t}{\sqrt{2-k^2}}\right)\biggr),\quad
k\in(0,1),
\end{align*}
inside each of them, and a one-parameter periodic orbits
\begin{align*}
\tilde{q}^k(t)
 =&\biggl(\frac{\sqrt{2}k}{\sqrt{2k^2-1}}\cn\left(\frac{t}{\sqrt{2k^2-1}}\right),\\
& \quad -\frac{\sqrt{2}k}{2k^2-1}\sn\left(\frac{t}{\sqrt{2k^2-1}}\right) 
 \dn\left(\frac{t}{\sqrt{2k^2-1}}\right)\biggr),\quad
k\in\bigl(1/\sqrt{2},1\bigr),
\end{align*}
outside of them, as shown in Fig.~\ref{fig:5a}(a),
 where $\sn$, $\cn$ and $\dn$ represent the Jacobi elliptic functions with the elliptic modulus $k$.
{See \cite{BF54} for general information on elliptic functions.}
The periods of $q^k_{\pm}(t)$ and $\tilde{q}^k(t)$ are given
 by $T^k=2K(k)\sqrt{2-k^2}$ and $\tilde{T}^k=4K(k)\sqrt{2k^2-1}$, respectively,
 where $K(k)$ is the complete elliptic integral of the first kind.
See also \cite{GH83,W90}.

Assume that the resonance conditions
\begin{equation}
lT^k=\frac{2\pi m}{\omega},\quad\mbox{i.e.,}\quad
\omega=\frac{2\pi m}{2lK(k)\sqrt{2-k^2}},
\label{eqn:resk}
\end{equation}
and
\begin{equation}
l\tilde{T}^k=\frac{2\pi m}{\omega},\quad\mbox{i.e.,}\quad
\omega=\frac{2\pi m}{4lK(k)\sqrt{2k^2-1}},
\label{eqn:tresk}
\end{equation}
hold for $q_\pm^k(t)$ and $\tilde{q}^k(t)$, respectively,
 with $l,m>0$ relatively prime integers.
We compute the subharmonic Melnikov function \eqref{eqn:subM}
 for $q_\pm^k(t)$ and $\tilde{q}^k(t)$ as
\[
M_\pm^{m/l}(\tau)=-\delta J_1(k,l)\pm\beta J_2(k,m,l)\sin\tau
\]
and
\[
\tilde{M}^{m/l}(\tau)=-\delta\tilde{J}_1(k,l)+\beta\tilde{J}_2(k,m,l)\sin\tau,
\]
respectively, where
\begin{align*}
&
J_1(k,l)=\frac{4l[(2-k^2)E(k)-2k'^2K(k)]}{3(2-k^2)^{3/2}},\\
&
J_2(k,m,l)=
\begin{cases}
\sqrt{2}\pi \omega\sech\left(\displaystyle\frac{m\pi K(k')}{K(k)}\right) & \mbox{(for $l=1$)};\\
0\quad & \mbox{(for $l\neq 1$)},\
\end{cases}\\
&
\tilde{J}_1(k,l)=\frac{8l[(2k^2-1)E(k)+k'^2K(k)]}{3(2k^2-1)^{3/2}},\\
&
\tilde{J}_2(k,m,l)=
\begin{cases}
2\sqrt{2}\pi \omega\sech\left(\displaystyle\frac{m\pi K(k')}{2K(k)}\right)
 & \mbox{(for $l=1$ and $m$ odd)}; \\
0 & \mbox{(for $l\neq 1$ or $m$ even).}
\end{cases}
\end{align*}
Here $E(k)$ is the complete elliptic integral of the second kind 
 and $k'=\sqrt{1-k^2}$ is the complimentary elliptic modulus.
We see that the subharmonic Melnikov functions $M_\pm^{m/l}(\tau)$ and $\tilde{M}^{m/l}(\tau)$
 are not identically zero on any connected open set in $\Sset_T^1$.
We also compute the homoclinic Melnikov function \eqref{eqn:homM} for $q_\pm^\h(t)$ as
\[
M_\pm(\tau)
 =-\frac{4}{3}\delta\pm\sqrt{2}\pi\omega\beta\csch\left(\frac{\pi\omega}{2}\right)\sin\tau,
\]
which is not identically zero on any connected open set in $\Sset_T^1$.
See also \cite{GH83,W90} for the computations of the Melnikov functions.

Let
\begin{align*}
&
R=\{k\in(0,1)\mid \mbox{$k$ satisfies \eqref{eqn:resk} for $m,l\in\Nset$}\},\\
&
\tilde{R}=\bigl\{k\in\bigl(1/\sqrt{2},1\bigr)\mid
\mbox{$k$ satisfies \eqref{eqn:tresk} for $m,l\in\Nset$}\bigr\},
\end{align*}
and let
\begin{align*}
&S_\pm^k=\{(x,\theta)\in\Rset^2\times\Sset_T^1\mid x=q_\pm^k(t),\},\\
&
\tilde{S}^k=\{(x,\theta)\in\Rset^2\times\Sset_T^1\mid x=\tilde{q}^k(t)\},\\
&
S_\pm^\h=\{(x,\theta)\in\Rset^2\times\Sset_T^1\mid x=q_\pm^\h(t)\}.
\end{align*}
Applying Theorems~\ref{thm:sub-mel} and \ref{thm:hom-mel},
 we obtain the following.

\begin{prop}
\label{prop:5a}
The first integral \eqref{Ham-Duffing} does not persist
 near $S_\pm^k$ for $k\in R$, $\tilde{S}^k$ for $k\in\tilde{R}$, and $S_\pm^\h$
 in \eqref{Duffing} with $a=1$ for $\epsilon>0$.
\end{prop}

\begin{rmk}
When $\beta>0$ but $\delta=0$, so that Eq.~\eqref{Duffing} is Hamiltonian,
 the statement of Proposition~\ref{prop:5a} {still} holds
 near $S^k_\pm$ for $k\in R_1$, $\tilde{S}^k$ for $k\in\tilde{R}_\mathrm{o}$,
 and $S_\pm^\h$, where
\begin{align*}
&
R_1=\{k\in(0,1)\mid \mbox{$k$ satisfies \eqref{eqn:resk} with $l=1$}\},\\
&
\tilde{R}_\mathrm{o}=\bigl\{k\in\bigl(1/\sqrt{2},1\bigr)\mid
 \mbox{$k$ satisfies \eqref{eqn:tresk} with $l=1$ and $m$ odd}\bigr\},
\end{align*}
\end{rmk}

We turn to the case of $a=-1$.
When $\epsilon=0$, in the phase plane
 there exists a one-parameter family of periodic orbits
\begin{align*}
{\gamma^k(t)}
 =&\biggl({\frac{\sqrt{2}k}{\sqrt{1-2k^2}}}\cn\left(\frac{t}{\sqrt{1-2k^2}}\right),\\
& \quad-\frac{\sqrt{2}k}{1-2k^2}\sn\left(\frac{t}{\sqrt{1-2k^2}}\right)
\dn\left(\frac{t}{\sqrt{1-2k^2}}\right)\biggr),\quad
k\in\bigl(0,1/\sqrt{2}\bigr),
\end{align*}
as shown in Fig.~\ref{fig:5a}(b),
 and their period is given by $\hat{T}^k=4K(k)\sqrt{1-2k^2}$.
See also \cite{Y94,Y96}.
Assume that the resonance conditions
\begin{equation}
l\hat{T}^k=\frac{2\pi m}{\omega},\quad\mbox{i.e.,}\quad
\omega=\frac{\pi m}{2lK(k)\sqrt{1-2k^2}}
\label{eqn:hresk}
\end{equation}
holds for $l,m>0$ relatively prime integers.
We compute the subharmonic Melnikov function \eqref{eqn:subM}
 for $\gamma^k(t)$ as
\[
\hat{M}^{m/l}(\tau)=-\delta\hat{J}_1(k,l)\pm\beta\hat{J}_2(k,m,l)\sin\tau,
\]
where
\begin{align*}
&
\hat{J}_1(k,l)=\frac{8l[(2k^2-1)E(k)+k'^2K(k)]}{3(1-2k^2)^{3/2}},\\
&
\hat{J}_2(k,m,l)=
\begin{cases}
\displaystyle
\frac{\sqrt{2}\pi^2m}{K(k)\sqrt{1-2k^2}}
\sech\left(\frac{\pi mK(k')}{2K(k)}\right) & \mbox{(for $l=1$ and $m$ odd)};\\
0\quad & \mbox{(for $l\neq 1$ or $m$ even)}.
\end{cases}
\end{align*}
See also \cite{Y94,Y96} for the computations of the Melnikov function.
Thus, the Melnikov function $\bar{M}^{m/l}(\tau)$
 is not identically zero on any connected open set in $\Sset_T^1$.

Let
\[
\hat{R}=\bigl\{k\in\bigl(0,1/\sqrt{2}\bigr)\mid
 \mbox{$k$ satisfies \eqref{eqn:hresk} for $m,l\in\Nset$}\bigr\}
\]
and let
\[
\hat{S}^k=\{(x,\theta)\in\Rset^2\times\Sset_T^1\mid x=\gamma^k(t)\}.
\]
Applying Theorem~\ref{thm:sub-mel}, we obtain the following.

\begin{prop}
The first integral \eqref{Ham-Duffing} does not persist
 near $\hat{S}^k$ for $k\in\hat{R}$ in \eqref{Duffing} with $a=-1$ for $\epsilon>0$.
\end{prop}

\begin{rmk}
When $\beta>0$ but $\delta=0$, i.e., Eq.~\eqref{Duffing} is Hamiltonian,
 the statement of Proposition~\ref{prop:5b} {still} holds
 near $\hat{S}^k$ for $k\in\hat{R}_\mathrm{o}$, where
\[
\hat{R}_\mathrm{o}=\bigl\{k\in\bigl(0,1/\sqrt{2}\bigr)\mid
 \mbox{$k$ satisfies {\eqref{eqn:hresk}} with $l=1$ and $m$ odd}\bigr\}.
\]
\end{rmk}

{
\subsection{Two pendula coupled with a harmonic oscillator}

\begin{figure}
\includegraphics[scale=1.2,bb=0 0 108 106]{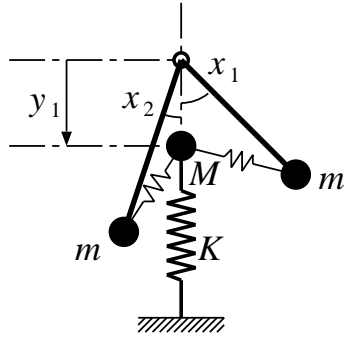}
\caption{Two identical pendula coupled with a harmonic oscillator.
\label{fig:ex2}}
\end{figure}

We next consider the three-degree-of-freedom Hamiltonian system
\begin{equation}
\begin{split}
&
\dot{x}_1=x_3,\quad
\dot{x}_3=-\sin x_1-\epsilon y_1\sin x_1,\\
&
\dot{x}_2=x_4,\quad
\dot{x}_4=-\sin x_2-\epsilon y_1\sin x_2,\\
&
\dot{y}_1=y_2,\quad
\dot{y}_2=-\omega_0^2y_1+\epsilon(\cos x_1+\cos x_2)
\end{split}
\label{eqn:ex2}
\end{equation}
with the Hamiltonian
\[
H=-\cos x_1-\cos x_2+\frac{1}{2}(x_3^2+x_4^2+\omega_0^2y_1^2+y_2^2)
 -\epsilon y_1(\cos x_1+\cos x_2),
\]
where $x_1,x_2\in\Sset_{2\pi}^1$, $x_3,x_4,y_1,y_2\in\Rset$
 and $\omega_0$ is a positive constant.
The system~\eqref{eqn:ex2}
 represents non-dimensionalized equations of motion for two identical pendula coupled
 with a harmonic oscillator shown in Fig.~\ref{fig:ex2}.
Here the gravitational force acts downwards,
 and the spring $K$ generates a restoring force $Ky_1$,
 where $y_1$ is the displacement of the mass $M=m$
 from the pivot of the pendula.
Linear restoring forces with a spring constant of $O(\epsilon)$ and zero natural length also occur
 between the two masses $m$ and the mass $M$.
In particular, $\omega_0^2=K\ell/Mg+O(\epsilon)$,
 where $g$ is the gravitational acceleration and $\ell$ is the length from the pivot to the mass $m$.

Introduce the action-angle coordinates $(I,\theta)\in\Rset_+\times\Sset_{2\pi}^1$
 such that
\[
y_1=\sqrt{\frac{2I}{\omega_0}}\sin\theta,\quad
y_2=\sqrt{2\omega_0I}\cos\theta
\]
and rewrite \eqref{eqn:ex2} as
\begin{equation}
\begin{split}
&
\dot{x}_1=x_3,\quad
\dot{x}_3=-\sin x_1-\epsilon\sqrt{\frac{2I}{\omega_0}}\sin\theta\sin x_1,\\
&
\dot{x}_2=x_4,\quad
\dot{x}_4=-\sin x_2-\epsilon\sqrt{\frac{2I}{\omega_0}}\sin\theta\sin x_2,\\
&
\dot{I}=\epsilon\sqrt{\frac{2I}{\omega_0}}\cos\theta(\cos x_1+\cos x_2),\\
&
\dot{\theta}=\omega_0-\epsilon\frac{\sin\theta}{\sqrt{2\omega_0 I}}(\cos x_1+\cos x_2),
\end{split}
\label{eqn:ex2a}
\end{equation}
which has the form \eqref{g-mel} with
\begin{align*}
&
H^0(x,I)=-\cos x_1-\cos x_2+\frac{1}{2}(x_3^2+x_4^2)+\omega_0 I,\\
&
H^1(x,I,\theta)=-\sqrt{\frac{2I}{\omega_0}}\sin\theta(\cos x_1+\cos x_2).
\end{align*}
where $\Rset_+$ denotes the set of nonnegative real numbers.
When $\epsilon=0$, the $x$-component of \eqref{eqn:ex2a}
 has a first integral
\[
F_2(x,I)=-\cos x_1+\frac{1}{2}x_3^2
\]
and a hyperbolic equilibrium $x^I=(\pi,\pi,0,0)$
 to which there exist four one-parameter families of homoclinic orbits
\begin{align*}
&
q_{\pm,+}^I(t;\alpha)=(\pm2\arcsin(\tanh t),2\arcsin(\tanh(t+\alpha)),
 \pm2\sech t,2\sech(t+\alpha)),\\
&
q_{\pm,-}^I(t;\alpha)=(\pm2\arcsin(\tanh t),-2\arcsin(\tanh (t+\alpha))
 \pm2\sech t,-2\sech(t+\alpha)),
\end{align*}
where $\alpha\in\Rset$.
Thus, assumptions~(W1)-(W3) hold with $m=2$.

We compute \eqref{eqn:barMk}
 for the homocloinic orbits $(x,I,\theta)=(q_{\pm,\pm}^I(t;\alpha),I,\omega_0t+\theta_0)$ as
\begin{align*}
\bar{M}_2^I(\theta_0,\alpha)
=& -\sqrt{\frac{2I}{\omega_0}}\int_{-\infty}^\infty 2\sin(\omega_0 t+\theta_0)
 \sech t\,\sin(2\arcsin(\tanh t))dt\\
=& -4\sqrt{\frac{2I}{\omega_0}}\,\cos\theta_0\int_{-\infty}^\infty\sech^2 t\,\tanh t\,\sin\omega_0t\,dt\\
=& -\pi\sqrt{8\omega_0^3I}\,\csch\left(\frac{\pi\omega_0}{2}\right)\,\cos\theta_0.
\end{align*}
On the other hand,
 letting $\{T_j^{I,\alpha}\}_{j=-\infty}^\infty$ be a time sequence satisfying \eqref{eqn:Tj},
 we write the integral in \eqref{eqn:barM1} as
\begin{align*}
&
-\sqrt{\frac{2I}{\omega_0}}\int_{T_{-j}^{I,\alpha}}^{T_j^{I,\alpha}}\cos(\omega_0 t+\theta_0)
 (\cos(2\arcsin(\tanh t))\\
& \qquad
 +\cos(2\arcsin(\tanh(t+\alpha))))dt\\
&=-\sqrt{\frac{2I}{\omega_0}}
\biggl(\cos\theta_0
 \int_{T_{-j}^{I,\alpha}}^{T_j^{I,\alpha}}(1-2\tanh^2t)\cos\omega_0t\,dt\\
&\qquad
 -\sin\theta_0
 \int_{T_{-j}^{I,\alpha}}^{T_j^{I,\alpha}}(1-2\tanh^2t)\sin\omega_0t\,dt\\
 &\qquad
 +\cos(\theta_0-\alpha\omega)
 \int_{T_{-j}^{I,\alpha}+\alpha}^{T_j^{I,\alpha}+\alpha}(1-2\tanh^2t)\cos\omega_0t\,dt\\
 &\qquad
  -\sin(\theta_0-\alpha\omega)
 \int_{T_{-j}^{I,\alpha}+\alpha}^{T_j^{I,\alpha}+\alpha}(1-2\tanh^2t)\sin\omega_0t\,dt\biggr).
\end{align*}
Since
\[
\lim_{j\to\pm\infty}\omega_0 T_j^{I,\alpha}=0\mod 2\pi,
\]
we have
\begin{align*}
&
\lim_{j\to\infty}\int_{T_{-j}^{I,\alpha}}^{T_j^{I,\alpha}}(1-2\tanh^2t)\cos\omega_0t\,dt\\
&=\lim_{j\to\infty}\int_{T_{-j}^{I,\alpha}}^{T_j^{I,\alpha}}2(1-\tanh^2t)\,\cos\omega_0t\,dt
-\lim_{j\to\infty}\int_{T_{-j}^{I,\alpha}}^{T_j^{I,\alpha}}\cos\omega_0t\,dt\\
&=2\pi\omega_0\csch\left(\frac{\pi\omega_0}{2}\right)
\end{align*}
and
\begin{align*}
&
\lim_{j\to\infty}\int_{T_{-j}^{I,\alpha}}^{T_j^{I,\alpha}}(1-2\tanh^2t)\sin\omega_0t\,dt\\
&=\lim_{j\to\infty}\int_{T_{-j}^{I,\alpha}}^{T_j^{I,\alpha}}2(1-\tanh^2t)\,\sin\omega_0t\,dt
-\lim_{j\to\infty}\int_{T_{-j}^{I,\alpha}}^{T_j^{I,\alpha}}\sin\omega_0t\,dt
=0.
\end{align*}
Hence, we obtain
\begin{align*}
\bar{M}_1^I(\theta_0,\alpha)
=-\pi\sqrt{8\omega_0^3 I}\,\csch\left(\frac{\pi\omega_0}{2}\right)
 (\cos\theta_0+\cos(\theta_0-\omega_0\alpha)).
\end{align*}\color{black}
We see that $\bar{M}_k^I(\theta_0,\alpha)$, $k=1,2$,
 are not identically zero on any connected open set in $\Rset_+\times\Sset_{2\pi}^1\times\Rset$.
Applying Theorem~\ref{thm:hom-mel2}, we obtain the following.}
\begin{prop}
\label{prop:ex2}
The first integrals $F_2(x,I)$ and $I$ do not persist near
\[
\bar{\Gamma}^\h
 =\bar{\Gamma}_{+,+}^\h\cup\bar{\Gamma}_{+,-}^\h\cup\bar{\Gamma}_{-,+}^\h\cup\bar{\Gamma}_{-,-}^\h
\]
in \eqref{eqn:ex2a} for $\epsilon>0$, where
\[
\bar{\Gamma}_{\pm,\pm}^\h
 =\{(q_{\pm,\pm}^I(t;\alpha),I,\theta)\mid t\in\Rset,I\in\Rset_+,\theta\in\Sset_{2\pi}^1\}.
\]
\end{prop}

\begin{rmk}
We have
\[
\det D\bar{M}^I(\theta,\alpha)
 =-4\pi^2\omega_0^2 I\,\csch^2\left(\frac{\pi\omega_0}{2}\right)
 \sin\theta_0\,\sin(\theta_0-\omega\alpha).
\]
Hence, if $\bar{M}^I(\theta,\alpha)=0$, then $\det D\bar{M}^I(\theta,\alpha)\neq 0$.
From Theorem~\ref{thm:W} we see that the stable and unstable manifolds
 of the perturbed periodic orbit near $\gamma_I^\p=\{(x^I,I,\theta)\mid\theta\in\Sset_{2\pi}^1\}$
 intersect transversely on its level set for $\epsilon>0$ sufficiently small.
\end{rmk}

\subsection{Periodically forced rigid body}

\begin{figure}
\includegraphics[scale=1,bb=0 0 196 120]{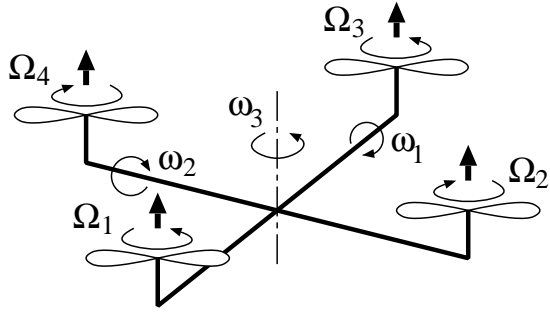}
\caption{Mathematical model for a quadrotor helicopter.
\label{fig:5b}}
\end{figure}

We next consider a three-dimensional system
\begin{equation}
\begin{split}
&
\dot{\omega}_1
=\frac{I_2-I_3}{I_1}\omega_2 \omega_3
 -\frac{I_0}{I_1}\Omega\omega_2+\frac{\ell b}{I_1}V_1,\\
&
\dot{\omega}_2
=\frac{I_3-I_1}{I_2}\omega_3 \omega_1
 +\frac{I_0}{I_2}\Omega\omega_1+\frac{\ell b}{I_2}V_2,\\
&
\dot{\omega}_3
=\frac{I_1-I_2}{I_3}\omega_1 \omega_2
+\frac{\ell d}{I_3}V_3,
\end{split}
\label{RB}
\end{equation}
which provides a mathematical model for a quadrotor helicopter shown in Fig.~\ref{fig:5b}.
In the model, $\omega_j{\in\Rset}$ and $I_j{>0}$, $j=1,2,3$, respectively,
 denote the angular velocities and moments of inertia about the quadrotor's principal axes, 
 $\ell$ represents the length from the center of mass to the rotational axis of the rotor,
 and $I_0$, $b$ and $d$ represent the rotor's moment of inertia about the rotational axis,
 thrust factor and drag factor, respectively.
Moreover,
\[
\Omega=\Omega_2+\Omega_4-\Omega_1-\Omega_3
\]
and
\[
V_1=\Omega_4^2-\Omega_2^2,\quad
V_2=\Omega_3^2-\Omega_1^2,\quad
V_3=\Omega_2^2+\Omega_4^2-\Omega_1^2-\Omega_3^2,
\]
where $\Omega_j$ is the angular velocity of the $j$th rotor for $j=1$-$4$.
See \cite{BMS04,HMLO02} for the derivation of \eqref{RB}.
{In particular}, the quadrotor can hover only if
\[
\Omega_j=\Omega_0:=\frac{1}{2}\sqrt{\frac{m_0 g}{b}},\quad
\mbox{$j=1$-$4$},
\]
where $m_0$ and $g$ are, respectively, the quadrotor's mass and gravitational acceleration.

Let $T>0$ be a constant,
 and let $\Omega_j=\Omega_0+\epsilon\Delta\Omega_j(t)$,
 where $\Delta\Omega_j(t)$ is a $T$-periodic function, for $j=1$-$4$.
This corresponds to a situation
 in which the quadrotor is subjected to periodic perturbations when hovering.
Let
\begin{align*}
&
v_1(t)=\Delta\Omega_4(t)-\Delta\Omega_2(t),\quad
v_2(t)=\Delta\Omega_3(t)-\Delta\Omega_1(t), \\
&
v_3(t)=\Delta\Omega_4(t)+\Delta\Omega_2(t)-\Delta\Omega_3(t)-\Delta\Omega_1(t)
\end{align*}
and
\[
\beta_0=I_0\Omega_0,\quad
\beta_1=\beta_2=2\ell b\Omega_0^2,\quad
\beta_3=2\ell d\Omega_0^2.
\]
Equation~\eqref{RB} is written as
\begin{equation}
\begin{split}
&
\dot{\omega}_1
=\frac{I_2-I_3}{I_1}\omega_2 \omega_3
 +\varepsilon\left(-\frac{\beta_0}{I_1}v_3(t)\omega_2+\frac{\beta_1}{I_1}v_1(t)\right)+O(\epsilon^2),\\
&
\dot{\omega}_2
=\frac{I_3-I_1}{I_2}\omega_3 \omega_1
 +\varepsilon\left(\frac{\beta_0}{I_2}v_3(t)\omega_1+\frac{\beta_2}{I_2}v_2(t)\right)+O(\epsilon^2),\\
&
\dot{\omega}_3
=\frac{I_1-I_2}{I_3}\omega_1 \omega_2
 +\varepsilon\frac{\beta_3}{I_3}v_3(t)+O(\epsilon^2),
\end{split}
\label{PFRB}
\end{equation}
in which chaotic motions were discussed in \cite{Y18}
 when $\beta_1=0$ and $v_2(t)=v_3(t)=\sin\nu t$ with $\nu>0$ a constant.
When $\epsilon=0$, Eq.~\eqref{PFRB} has a first integral
\[
F(\omega)=\frac{1}{2}(I_1\omega_1^2+I_2\omega_2^2+I_3\omega_3^2)
\]
and nonhyperbolic equilibria at
\[
p_{1\pm}(c_1)=(\pm c_1, 0, 0),\quad
p_{2\pm}(c_2)=(0,\pm c_2, 0),\quad
p_{3\pm}(c)=(0,0, \pm c_3)
\]
on the level set $F(\omega)=c>0$, where $c_j=\sqrt{2c/I_j}$, $j=1,2,3$.
The first integral $F(\omega)$ corresponds to the (Hamiltonian) energy of the unperturbed rigid body.

Let $X_\epsilon(\omega,t)=X^0(\omega)+\epsilon X^1(\omega,t)+O(\epsilon^2)$
 denote the non-autonomous vector field of \eqref{PFRB}
 and define the corresponding autonomous vector field
 $\tilde{X}_\epsilon(\omega,\theta)
 =\tilde{X}^0(\omega,\theta)+\epsilon\tilde{X}^1(\omega,\theta)+O(\epsilon^2)$
 on $\Rset^3\times\Sset_T^1$ like \eqref{mel-aut}, where
\begin{align*}
\tilde{X}^0(\omega,\theta)=
\begin{pmatrix}
X^0(\omega)\\
1
\end{pmatrix},\quad
\tilde{X}^1(\omega,\theta)=
\begin{pmatrix}
X^1(\omega,\theta)\\
1
\end{pmatrix}.
\end{align*}
The unperturbed vector field $\tilde{X}^0(\omega,\theta)$ has six one-parameter families
 of nonhyperbolic periodic orbits $\gamma_{j\pm,c_j}(t)=(p_{j\pm}(c_j), t)$, $j=1,2,3$.
We compute the integral \eqref{integral1} as
\[
\mathscr{I}_{F,\gamma_{j\pm,c_j}}
=\int_0^T dF(\tilde{X}^1)(p_j(c_j),t)dt
=\pm c_j\beta_j \int_0^T v_j(t)dt,\quad
j=1,2,3,
\]
and apply Theorems~\ref{thm:main2} and \ref{thm:main1} to obtain the following.
 
\begin{prop}
\label{prop:5b}
For $j=1,2,3$, if 
\[
\beta_j\int_0^T v_j(t)\neq 0,
\]
then the periodic orbit $\gamma_{j\pm,c_j}(t)
 $ does not persist
 for any $c_j>0$ and the first integral $F(\omega)$ does not persist near
\[
\{(p_{j+}(c_j),\theta)\in\Rset \mid c_j>0,\theta\in\Sset_T^1\}
\cup\{(p_{j-}(c_j),\theta)\in\Rset \mid c_j>0,\theta\in\Sset_T^1\}
\]
in \eqref{PFRB}.
\end{prop}

\begin{rmk}\
\begin{enumerate}
\setlength{\leftskip}{-1.8em}
\item[(i)]
In \cite{Y18},
 when $\beta_1=0$ and $v_2(t)=v_3(t)=\sin\nu t$ with $\nu>0$ a constant,
 it was shown that the periodic orbits $\gamma_{2\pm,c_2}$ persist for $c_2>0$
 if and only if $\beta_0=0$ or $\beta_3=0$ (see Proposition~2 of \cite{Y18}).
\item[(ii)]
The unperturbed vector field $X^0(\omega)$ has another first integral
\[
\tilde{F}(\omega)=(I_1^2\omega_1^2+I_2^2\omega_2^2+I_3^2\omega_3^2),
\]
which corresponds to the angular momentum of the rigid body.
We compute the integral \eqref{integral1} as
\[
\mathscr{I}_{\tilde{F},\gamma}
=\int_0^T d\tilde{F}(\hat{X}^1)(p_j(c_j),t)dt
=\pm 2c_jI_j\beta_j \int_0^T v_j(t)dt,\quad
j=1,2,3,
\]
so that the same statement  as Proposition~\ref{prop:5b} holds for $\tilde{F}(\omega)$.
\end{enumerate}
\end{rmk}

\subsection{Three-mode truncation of a buckled beam}

\begin{figure}
\includegraphics[scale=0.9,bb=0 0 96 183]{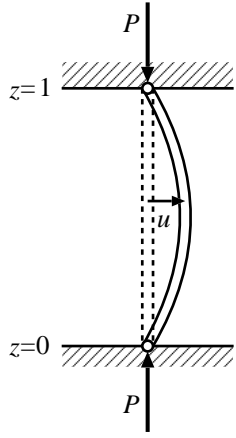}
\caption{Buckeled beam.
The variables $u$ and $P$ represent the deflection and compressive force, respectively.
The length of the beam when $u\equiv 0$ is non-dimensionalized to the unity.
\label{fig:5c}}
\end{figure}

Finally, we consider a six-dimensional autonomous system
\begin{equation}
\begin{split}
&
\dot{x}_1=x_4,\quad
\dot{x}_4=x_1-(x^2_1+\beta_1x_2^2+\beta_2x_3^2)x_1,\\
&
\dot{x}_2=x_5,\quad
\dot{x}_5=-\omega^2_1x_2-\beta_1(x^2_1+\beta_1x_2^2+\beta_2x_3^2)x_2,\\
&
\dot{x}_3=x_6,\quad
\dot{x}_6=-\omega^2_2x_3-\beta_2(x^2_1+\beta_1x_2^2+\beta_2x_3^2)x_3,
\end{split}
\label{MODE}
\end{equation}
which represents a three-mode truncation of a buckled beam shown in Fig.~\ref{fig:5c},
 where {$x_j\in\Rset$, $j=1$-$6$,}
 $\omega_j,\beta_j>0$, $j=1,2$, are constants such that $\omega_1<\omega_2$.
See \cite{Y01} for the details on the model.
In \eqref{MODE} there is a saddle-center equilibrium at {$(x_1, \ldots, x_6)=(0,\ldots,0)$}
 and it has a homoclinic orbit.
It was also shown in \cite{Y05} that for almost all pairs of $\beta_1,\beta_2>0$
 the system \eqref{MODE} exhibits chaotic motions and it is nonintegrable.

Let $x_j=\sqrt{\epsilon}y_j$, $j=1$-$6$, with the small parameter $\epsilon$.
We rewrite \eqref{MODE} as
\begin{equation}
\begin{split}
&
\dot{y}_1=y_4,\quad
\dot{y}_4=y_1-\epsilon(y^2_1+\beta_1y_2^2+\beta_2y_3^2)y_1,\\
&
\dot{y}_2=y_5,\quad
\dot{y}_5
 =-\omega^2_1y_2-\epsilon\beta_1(y_1^2+\beta_1y_2^2+\beta_2y_3^2)y_2,\\
&
\dot{y}_3=y_6,\quad
\dot{y}_6
=-\omega^2_2y_3-\epsilon\beta_2(y_1^2+\beta_1y_2^2+\beta_2 y_3^2)y_3,
\end{split}
\label{MODEe}
\end{equation}
which is regarded as as a perturbation of a linear system.
When $\epsilon=0$, Eq.~\eqref{MODEe} has two one-parameter families of periodic orbits
{
\begin{align*}
&
\gamma_{1,c}(t)=(0,c\sin\omega_1 t,0,0,c\omega_1\cos\omega_1 t,0),\\
&
\gamma_{2,c}(t)=(0,0,c\sin\omega_2 t,0,0,c\omega_2\cos\omega_2 t)
\end{align*}
}
for $c>0$, three first integrals
\begin{align*}
F_1(y)=-y_1^2+y_4^2,\quad 
F_2(y)=\omega^2_1y_2^2+y_5^2,\quad
F_3(y)=\omega^2_2y_3^2+y_6^2,
\end{align*}
and six commutative vector fields
{
\begin{align*}
&
Z_1=(y_1,0,0,y_4,0,0),\quad
Z_2=(y_4,0,0,y_1,0,0),\\
&
Z_3=(0,y_2,0,0,y_5,0),\quad
Z_4=(0,y_5,0,0,-\omega_1^2y_2,0),\\
&
Z_5=(0,0,y_3,0,0,y_6),\quad
Z_6=(0,0,y_6,0,0,-\omega_2^2y_3).
\end{align*}
}
Moreover, the AVE of \eqref{MODEe} with $\epsilon=0$ 
 is given by
\begin{align*}
&
\dot{\eta}_1=-\eta_4,\quad
\dot{\eta}_2=\omega^2_1\eta_5,\quad
\dot{\eta}_3=\omega^2_2\eta_6,\\
&
\dot{\eta}_4=-\eta_1,\quad
\dot{\eta}_5=-\hat{\eta}_2,\quad
\dot{\eta}_6=-\eta_3,
\end{align*}
which has four linearly independent periodic solutions
{
\begin{align*}
&
\tilde{\gamma}_1(t)=(0,\omega_1\sin\omega_1 t,0,0,\cos\omega_1 t,0),\\
&
\tilde{\gamma}_2(t)=(0,\omega_1\cos\omega_1 t,0,0,-\sin\omega_1 t,0),\\
&
\tilde{\gamma}_3(t)=(0,0,\omega_2\sin\omega_2 t,0,0,\cos\omega_2 t),\\
&
\tilde{\gamma}_4(t)=(0,0,\omega_2\cos\omega_2 t,0,0,-\sin\omega_2 t)
\end{align*}
}
and two linearly independent unbounded solutions. 
We compute \eqref{integral1} and \eqref{integral2} as
\[
\mathscr{I}_{F_j,\gamma_{\ell,c}}
=\int_0^{2\pi/\omega_\ell}dF_j(X_1)(\gamma_{\ell,c}(t))dt=0,\quad
\mbox{$j=1,2,3$ and $\ell=1,2$,}
\]
and
\begin{align}
\mathscr{J}_{\tilde{\gamma}_j,Z_k,\gamma_{\ell,c}}
=& \int_0^{2\pi/\omega_\ell}\tilde{\gamma}_j(t)\cdot[X_1,Z_k]_{\gamma_{\ell,c}(t)}dt\notag\\
=&
\begin{cases}
\displaystyle
\frac{3}{2}\pi\beta^2_1c^3 & \mbox{if $(j,k,\ell)=(2,3,1)$};\\[1.5ex]
\displaystyle
\frac{3}{2}\pi\beta^2_2c^3 & \mbox{if $(j,k,\ell)=(4,5,2)$};\\
0 & \mbox{otherwise},
\end{cases}
\label{eqn:5.3}
\end{align}
where {$X^1$} represents the $O(\epsilon)$-terms of the vector field in \eqref{MODEe}.
In \eqref{eqn:5.3}, the subscript $j$ is allowed to take $1$ or $2$ for $\ell=1$,
 and $3$ or $4$ for $\ell=2$.
Theorems~\ref{thm:main2}, \ref{thm:main1} and  and \ref{thm:main3o}
 give no meaningful information on persistence of periodic orbits and first integrals,
 but application of Theorem~\ref{thm:main3} yields the following.

\begin{prop}
The commutative vector fields $Z_3$ and $Z_5$ do not persist
 near the {$(y_2,y_5)$- and $(y_3,y_6)$-planes}, respectively, in \eqref{MODEe}.
Moreover, in \eqref{MODE}, near the origin, there is no commutative vector field
 which has the linear term
\[
{\tilde{Z}_3=(0,x_2,0,0,x_5,0)\quad\mbox{or}\quad
\tilde{Z}_5=(0,0,x_3,0,0,x_6).}
\]
\end{prop}

\begin{proof}
The first part immediately follows from application of Theorem~\ref{thm:main3}.
The second part is easily proven
 since a vector field having such a linear term for \eqref{MODE}
 is transformed to $Z_3+O(\epsilon)$ or $Z_5+O(\epsilon)$ for \eqref{MODEe}.
\end{proof}

\begin{rmk}\ 
\begin{enumerate}
\setlength{\leftskip}{-1.8em}
\item[(i)]
By the Lyapunov center theorem (e.g., Theorem 5.6.7 of \cite{AM78}),
 there exist two families of periodic orbits in \eqref{MODE}
 if $\omega_2/\omega_1,\omega_1/\omega_2\notin\Zset$.
Hence, the periodic orbits $\gamma_{j,c}$, $j=1,2$,
 persist in \eqref{MODEe} for such values of $\omega_j$, $j=1,2$, at least.
\item[(ii)]
As shown in \cite{Y05},
 the Hamiltonian system \eqref{MODE} is nonintegrable for almost all pairs of $\beta_j$, $j=1,2$.
Hence, the three first integrals $F_j(y)$, $j=1,2,3$,
 do not persist in \eqref{MODEe} for such values of $\beta_j$, $j=1,2$, at least.
\end{enumerate}
\end{rmk}

\section*{Acknowledgement}
This work was partially supported by the JSPS KAKENHI Grant Numbers JP17H02859 and JP19J22791.

 

\setcounter{equation}{0}
\renewcommand{\theequation}{\Alph{section}.\arabic{equation}}

\appendix

\section{
Some auxiliary materials for Section 3}\label{AppendixA}
In this appendix, we provide some prerequisites for Section 3:
Basic notions and facts on connections of vector bundles
 and linear differential equations.
Similar materials are found in \cite{CR88,IY08,MR01}.
See, e.g., \cite{BT82} for necessary information on vector bundles.

\subsection{Connections and 
horizontal sections}
We begin with connections of vector bundles and their horizontal sections.
Henceforth $M$ represents a $C^1$ $m$-dimensional manifold for $m\in\Nset$,
and $E$ represents a $C^1$ vector bundle of rank $r$ over $M$
with a projection $\pi:E\to M$
for some $m,r\in\Nset$.
Let $C(M)$ be a set of all $C^1$ $\Rset$-valued functions on $M$
 and let $C(M,E)$ be a set of all $C^1$ sections of $E$.
Let $T^\ast M$ be the cotangent bundle of $M$.
Note that $T^*M\otimes E$ is also a $C^1$ vector bundle.
We first give basic definitions.

\begin{dfn}
An $\Rset$-linear map
$$
\nabla:C(M, E)\to C(M, T^*M\otimes E)
$$
is called a \emph{connection} of the vector bundle $E$ 
if
\begin{equation}
\nabla (fs)=df\otimes s+f\nabla s
\label{eqn:defa1i}
\end{equation}
for any $f\in C(M)$ and $s\in C(M, E)$.
A section $s\in C(M,E)$ is said to be \emph{horizontal} for the connection $\nabla$
if $\nabla s=0$.
\end{dfn}

Let $U\subset M$ be an open neighborhood 
and let $\{e_j\}_{j=1}^r$ be a frame on $U$, so that
any section $s\in C(M, E)$ is expressed as
\begin{equation}
s=\sum_{j=1}^r s^j e_j
\label{eqn:section}
\end{equation}
on $U$ for some $s^j\in C(M)$ for $j=1,\ldots, r$.

\begin{dfn}
For each $i=1,\ldots,r$ we can write
\begin{equation}
\nabla e_i=\sum_j \theta_i^j\otimes e_j,
\label{eqn:defa2}
\end{equation}
where $\theta_i^j:M\to T^\ast M$, $j=1,\ldots,r$.
The $r\times r$ matrix $\theta=(\theta_i^j)$ is called
the \emph{connection form} of $\nabla$ on $U$
in the frame $\{e_j\}_{j=1}^r$.
\end{dfn}

Let $s\in C(M,E)$.
Using \eqref{eqn:defa1i} and \eqref{eqn:defa2}, we compute
 \begin{align*}
\nabla s
=&\sum_{j=1}^r \nabla(s^je_j)
=\sum_{j=1}^r(ds^j\otimes e_j +s^j\nabla e_j)\\
=&\sum_{i=1}^r ds^i\otimes e_i
+\sum_{i=1}^r s^i\biggl(\sum_{j=1}^r \theta_i^j \otimes e_j\biggr)
=\sum_{i=1}^r\biggl(ds^i + \sum_{j=1}^r s^j \theta_j^i\biggr)\otimes e_i.
\end{align*}
Hence, the condition for the section $s$ to be horizontal, $\nabla s=0$, is equivalent to
\begin{equation}\label{local-conn}
ds^i + \sum_{j=1}^r s^j \theta_j^i=0,\quad i=1, ..., r,
\end{equation}
on $U$.

\begin{dfn}
Let $ E^\ast$ be the dual bundle of $E$.
A connection $\nabla^*$ of $E^*$ given by
\begin{equation}
d\langle s, \alpha \rangle
=\langle \nabla s, \alpha\rangle 
+\langle s, \nabla^* \alpha \rangle
\label{eqn:defa1ii}
\end{equation}
for any $s\in C(M, E)$ and $\alpha\in C(M, E^*)$ is called a \emph{dual connection} of $\nabla$.
\end{dfn}

Let $\{e^j\}_{j=1}^r$ be the dual frame for the frame $\{e_j\}_{j=1}^r$, i.e.,
\begin{equation}
\langle e_i,e^j\rangle=\delta_{ij},\quad
i,j=1,\ldots,r,
\label{eqn:dualfr}
\end{equation}
where $\delta_{ij}$ is Kronecker's delta.
We have the following relation between 
connections and their dual connections.

\begin{prop}\label{dual-conn}
Let $\theta=(\theta_i^j)$ be the connection form of a connection $\nabla$ on $U$.
Then the connection form $\theta^*$ of the dual connection $\nabla^*$
is given by ${\theta_i^*}^j=-\theta_j^i$ on $U$.
\end{prop}

\begin{proof}
Using \eqref{eqn:defa1ii} and \eqref{eqn:dualfr}, we compute
 \[
 0=d\langle e_i, e^j\rangle
=\langle\nabla e_i, e^j\rangle+\langle e_i, \nabla^*e^j\rangle.
\]
Since by \eqref{eqn:defa2} and \eqref{eqn:dualfr}
\[
\langle\nabla e_i, e^j\rangle=\left\langle\sum_{k=1}^r\theta_i^k\otimes e_k,e^j\right\rangle
=\theta_i^j=\left\langle e_i, \sum_{k=1}^r\theta_k^j \otimes e^k\right\rangle,
\]
we obtain
\[
\left\langle e_i, \sum_{k=1}^r\theta_k^j \otimes e^k+\nabla^*e^j\right\rangle=0,\quad i,j=1, \ldots,r.
\]
Hence,
\[
\nabla^* e^j=\sum_{k=1}^r -\theta_k^j \otimes e^k
\]
for $j=1,\ldots,r$. 
\end{proof}

\subsection{Connections and linear differential equations}
Let $m=1$ and assume that
 the one-dimensional manifold $M$ is paracompact and connected.
We will see below that
a connection of the vector bundle $E$ defines a linear differential equation
and horizontal sections of the connection correspond to solutions to the differential equations.

Take an open neighborhood $U\subset M$ and its local coordinate $t\in\Rset$.
Let $\nabla$ be a connection
and let $s\in C(M,E)$ be a horizontal section of $\nabla$ given by \eqref{eqn:section}.
We write the connection form $\theta=(\theta_i^j)$ as
$$
\theta_i^j=a_{ij}(t) dt
$$
for some $a_{ij}(t)\in C(M)$.
Then Eq.~\eqref{local-conn} is expressed as
\begin{equation}
ds^i + \sum_{j=1}^r a_{ji}(t)s^j dt=0,\quad i=1, ..., r.
\label{eqn:ds}
\end{equation}

Let $A(t)=(A_{ij}(t))$ be an $r\times r$ matrix with  $A_{ij}(t):=-a_{ji}(t)$
and let $\hat{s}(t)=(s^1(t), ..., s^r(t))^\mathrm{T}$.
From \eqref{eqn:ds} we obtain a linear differential equation
\begin{equation}
\frac{d}{dt}\hat{s}(t)=A(t)\hat{s}(t).
\label{eqn:a2}
\end{equation}
Thus, the relation $\nabla s=0$ is locally represented by a linear differential equation.
Below we apply the above argument to the VE \eqref{VEconn} and AVE \eqref{adjVEconn}
to derive \eqref{VEloc} and \eqref{adjVEloc}, respectively.

\subsubsection{Derivation of \eqref{VEloc}}
We consider the VE \eqref{VEconn} and set $M=\Gamma$ and $E=T_\Gamma$ with $r=n$.
Choose the frame $\displaystyle
\left(\frac{\partial}{\partial x_1},\ldots,\displaystyle\frac{\partial}{\partial x_n}\right)$
and write
\[
X=\sum_{j=1}^n X_j \frac{\partial}{\partial x_j}
\]
locally.
We compute
\begin{align*}
\nabla \frac{\partial}{\partial x_i}
=&dt\otimes \mathcal{L}_X\biggl(\frac{\partial}{\partial x_i}\biggr)\bigg|_\Gamma=
dt\otimes\biggl[\sum_{j=1}^n X_j \frac{\partial}{\partial x_j}, \frac{\partial}{\partial x_i}
\biggr]_\Gamma\\
=&dt\otimes\biggl(-\sum_{j=1}^n \frac{\partial X_j}{\partial x_i} \frac{\partial}{\partial x_j}\biggr)
\bigg|_\Gamma
=-\sum_{j=1}^n\frac{\partial X_i}{\partial x_j}(\phi(t))dt\otimes \frac{\partial}{\partial x_i},\quad
i=1,\ldots,n,
\end{align*}
so that
\begin{equation}
\theta_i^j=-\frac{\partial X_j}{\partial x_i}(\phi(t))dt,\quad\mbox{i.e.,}\quad
A_{ij}(t)=\frac{\partial X_i}{\partial x_j}(\phi(t)),\quad
i,j=1,\ldots,n.
\label{eqn:a21}
\end{equation}
This yields \eqref{VEloc} along with \eqref{eqn:a2}.

\subsubsection{Derivation of \eqref{adjVEloc}}

We next consider the AVE \eqref{adjVEconn} in the setting of Section~A.2.1.
Choose the frame $(dx_1,\ldots,dx_n)$.
Using Proposition~\ref{dual-conn} and \eqref{eqn:a21}, we obtain
\[
\theta_i^{*j}=\frac{\partial X_i}{\partial x_j}(\phi(t))\quad\mbox{i.e.,}\quad
A_{ij}(t)=-\frac{\partial X_j}{\partial x_i}(\phi(t)),\quad
i,j=1,\ldots,n.
\]
This yields \eqref{adjVEloc} along with \eqref{eqn:a2}.


\end{document}